\documentclass{FKBGN-axSIH}

\newtheorem{theorem}{Theorem}
\newtheorem{proposition}[theorem]{Proposition}
\newtheorem{lemma}[theorem]{Lemma}
\newtheorem{corollary}[theorem]{Corollary}

\theoremstyle{definition}

\newtheorem{constructions}[theorem]{\noindent{\font\=cmssi10\Constructions}\bf}

\newtheorem{definitions}[theorem]{\font\=cmssi10\Definitions\bf}

\newtheorem{remark}[theorem]{\font\=cmssi10\Remark\bf}

\newcommand\Cal{\mathcal}
\newcommand\bosy{\boldsymbol}

\newcommand\romanbb[2]{\roman{#1}\kern.5mm\lower.7mm\hbox{\font\=msbm5\#2}\kern.35mm}% \romanbb aR gives \roman a_{fivepoint{mathbb R}}
\renewcommand\math[1]{\hbox{$\kern0.4mm#1\kern0.4mm$}}% the same as above but also putting an extra space of .4mm:s before and after
\newcommand\mathss[3]{\hbox{\kern0.17mm\kern.#1mm$#3$\kern.#2mm\kern0.17mm}}% the same as above but allowing to choose the space before and after; e.g. `` ... that \mathss23{x=y} holds ... '' is the same as `` ... that \œ$ \kern0.17mm\kern.2mm  x=y  \kern0.3mm\kern.17mm$  holds ... ''
\newcommand\aall[1]{\forall\kern.8mm{#1}\kern.2mm\,;}
\newcommand\eexi[1]{\exists\kern.9mm{#1}\kern.2mm\,;}
\newcommand\impss[2]{\kern.#1mm\kern.17mm\Rightarrow\kern.17mm\kern.#2mm}
\def\afr#1_#2{\mathfrak#1_{\kern.2mm\lower.15mm\hbox{\font\=cmr6\#2}}} % e.g., \afr x_0 = \mathfrak x_0
\def\Alf{\hbox{\font\=cmmi10 scaled\magstep1\\char'013}\kern0.15mm}% bigger alpha
\def\Eps{\hbox{\font\=cmmi10 scaled\magstep1\\char'017}\kern0.15mm}% bigger varepsilon
\def\Iota{\kern.15mm\hbox{\font\=cmmi10 scaled\magstep1\\char'023}\kern0.2mm}% bigger iota
\def\Nu{\hbox{\font\=cmmi10 scaled\magstep1\\char'027}\kern0.25mm}% bigger nu
\def\uvarPi{\kern.15mm\underline{\kern-.15mm\varPi\kern-.85mm}\kern.85mm}% underlined varPi
\def\uOmega{\kern.3mm\underline{\kern-.3mm\Omega\kern-.3mm}\kern.3mm}% underlined Omega
\def\rbrakf{\kern.9mm]\kern.2mm\lower.8mm\hbox{\font\=cmr5\f}\kern.2mm} % [f,g]_f : the function x mapsto (fx,gx) , e.g. $[\,\sp f\sp,\sp g\rbrakf$
\def\ftimes{\hbox{\kern-.2mm${}\times\kern-2.7mm\lower.9mm\hbox{\font\=cmr5\f}\kern1.7mm$}} % f\ftimes g : the function (x,y) mapsto (fx,gy)
\def\bbNo{\mathbb N\kern.15mm\lower.65mm\hbox{\font\=cmr6\0}\kern.1mm} % N_0 = \No = set of natural numbers
\def\sbbNo{\mathbb N\kern.07mm\lower.45mm\hbox{\font\=cmr5\0}\kern.1mm} % the preceding for sub- and superscripts
\def\bbR{\mathbb R} % = \Re = set of real numbers
\newcommand\ssbb[3]{\hspace{.#1mm}\mathbb#3\hspace{.#2mm}} % e.g. $(\ssbb42 R)$ gives $(\hspace{.4mm}\mathbb R\hspace{.2mm})$ , \ssbb00 R = \mathbb R
\def\lbb#1_#2{\mathbb#1\kern.2mm\raise.52mm\hbox{$_{_{#2}}$}} % e.g., to get \mathbb Z_- use \lbb Z_-
\def\rbb#1^#2{\mathbb#1\kern.2mm\lower.55mm\hbox{$^{^{#2}}$}} % e.g., to get \mathbb R^+ use \rbb R^+
\def\Rlplus{\mathbb R\kern.2mm\lower.33mm\hbox{\font\=cmr5\+}} % \Rlplus = R_+ = { t : t real and t \ge 0 } = \lbb R_+
\def\fbbR{\raise1.2mm\hbox{\font\=cmr5\f}\kern.3mm\mathbb R} % the standard field structure of real numbers
\def\tfbbR{\raise1.2mm\hbox{\font\=cmr5\tf}\kern.3mm\mathbb R} % standard topological field of real numbers
\def\stfbbR{\raise.5mm\hbox{\font\=cmr5\t\kern-.1mmf}\kern.3mm\mathbb R} % the preceding for sub- and superscripts
\def\tfbbC{\raise1.23mm\hbox{\font\=cmr5\tf}\kern.1mm\mathbb C} % standard topological field of complex numbers
\def\tfbbH{\raise1.2mm\hbox{\font\=cmr5\tf}\kern.3mm\mathbb H} % standard topological division ring of quaternions
\def\taubb_#1{\tau\kern-.15mm\lower.7mm\hbox{\font\=msbm5\#1}\kern.3mm} % tuottaa tau_{bb #1} , e.g. \taubb_R = the standard topology of the real line
\def\bartau_bb#1{\bar\tau\kern-.15mm\lower.7mm\hbox{\font\=msbm5\#1}\kern.3mm} % \bar\tau_{mathbb #1} , e.g. \bartaubb_R = the standard topology of the extended real line
\def\cinfty{\raise1.35mm\hbox{\font\=cmr5\c}\kern-.15mm\infty} % the complex infinity
\def\inftyyplus{{\vphantom{p_{p_p}}\infty\RHB{.5}{\fiveroman+}}} % \infty^+ for superscript, e.g. $C^{\,\inftyyplus}(O)$
\def\Lebmea^#1{\kern.25mm\mu_{\kern.3mm\hbox{\font\=cmr5\Leb}}^{\vphantom n\kern.5mm{#1}}} % \Lebmea^N , \Lebmea^{} = the complete Lebesgue measure on R^N , R , resp.
\newcommand\openIval[1]{\null\kern.35mm]\kern.7mm#1\kern.8mm[\kern.35mm\null} % \openIval{s,t} = ] s,t [

\def\setRC{\{\kern0.15mm\raise1.2mm\hbox{\font\=cmr5\tf}\kern.3mm\mathbb R\,,\kern-0.3mm\raise1.23mm\hbox{\font\=cmr5\tf}\kern.1mm\mathbb C\,\}} % = {^{tf}R,^{tf}C} , use e.g. $\bosy K\in\setRC$
\def\Reit#1{_{\lower.2mm\hbox{\kern.2mm\hskip.#1mm\font\=msbm5\R\kern.15mm\font\=cmr5\t}\kern.15mm}} % E_{\mathbb R\roman t} = E\Reit3 = the realification of a complex topological vector space E ; use e.g. also $F\Reit0,G\Reit0,H\Reit1,\varUpsilon\Reit0,\varPi\Reit2$ , the number n puts a space of .n mm:s before the subscript
\def\Ceit#1{_{\lower.2mm\hbox{\hskip.#1mm\font\=msbm5\C\kern.15mm\font\=cmr5\t}\kern.15mm}} % E_{\mathbb C\roman t} = E\Ceit3 = the complexification of a real topological vector space E ; use e.g. also $F\Ceit0,G\Ceit0,H\Ceit1,\varUpsilon\Ceit0,\varPi\Ceit2$ , the number n puts a space of .n mm:s before the subscript

\def\lfbb_#1{_{\lower.25mm\hbox{\kern.3mm\font\=msbm5\#1\kern.15mm}}} % e.g., C^i(O)\lfbb_C = the complex space of complex valued C^i functions on O

\def\fRe{\hbox{\font\=cmr9\f\kern.1mm}\roman{Re}\kern.75mm}% the real      part of a function
\def\fIm{\hbox{\font\=cmr9\f\kern.1mm}\roman{Im}\kern.65mm}% the imaginary part of a function
\def\vecc#1{\kern-.5mm\vec{\kern.5mm#1}}% \vec moved a bit backwords so that in (\vecc x) the arrow would not get too close to ')'
\def\TVS{\roman{TVS}\kern0.4mm}%
\def\LCS{\roman{LCS}\kern0.4mm}%
\def\BaS{\roman{BaS}\kern0.4mm}%
\def\HilbLCS{\roman{H\kern.15mm\lower.7mm\hbox{\font\=cmr5\ilb}\kern.25mmLCS}\kern0.4mm} % \HilbLCS(K) = the class of locally convex spaces over K with topology determined by a Hilbert norm
\def\dimHa{{\rm dim_{_{\kern.2mm Ha}}}}% the Hamel dimension of a structured vector space
\def\rajou{{}^{}{\Cal B}_{s\,}}% the set of bounded sets in a TVS
\def\Lis{\Cal L\kern.3mmis\,} % \Lis(E,F) = set of linear homeomorphisms E to F
\def\Linb{\Cal L\lower.7mm\hbox{\kern.1mm\font\=cmmi6\b}}% \Linb(E,F) gives \Cal L_b(E,F) = the topological vector space of continuous linear maps E to F with the topology of uniform convergence on bounded sets

\def\dualbeta{^{\kern0.4mm\prime}_{\kern-.2mm\raise.95mm\hbox{$_{_\beta}$}}} % E\dualbeta = E'_\beta = strong topological dual of E
\def\Nbh{\Cal N_{\font\=cmmi6\lower.15mm\hbox{\kern.1mm\bh\kern.15mm}}}% \Nhb(x,T) = set of T-neighborhoods of x
\def\Topma{\roman{{Top_{}}_{\hbox{\font\=cmr6\ma}}\kern.15mm}}
\def\prodsubtext#1{\prod\kern-0.3mm\lower.9mm\hbox{\font\=cmr5\#1}\kern.7mm}
\def\prodc{\prod{_{_{\kern-.3mm\bold c\kern.15mm}}}}% cartesian propduct of sets
\def\vsprod_#1_#2{\prod\kern-0.3mm{}_{_{\roman{#1}\sp{#2}\,}}} % for example, \vsprod_TVS_\Re produces \prod_{\roman{TVS} I\!\!R}
\def\vscoprod_#1_#2{\coprod\kern-0.3mm{}_{_{\roman{#1}\sp{#2}\,}}} % see the preceding
\def\expnota^#1]_#2{\,^{#1\,]{_{}}_{\roman{#2}}}} % F\expnota^\Omega_{tvs} = the topological product vector space of all functions Omega to F
\def\expnote^#1]#2{\kern.8mm^{#1}\kern.3mm\raise1.2mm\hbox{\font\=cmsy5\\char'143}\raise1.1mm\hbox{\font\=cmr5\#2}} %
\def\expar#1]#2{\kern1mm\raise1.5mm\hbox{\font\=cmsy5\\char'042}\kern.6mm#1\kern.8mm\raise1.4mm\hbox{\font\=cmsy5\\char'145}\raise1.3mm\hbox{\font\=cmr5\#2}} %
\def\co{\hbox{\font\=cmmi12\c}\kern.15mm\lower.15mm\hbox{$_{\rm o}$}}% \co(I,F) = the supremumnormable space of functions I to F with become small outside finite sets
\def\LL^#1{L\kern0.15mm\raise.4mm\hbox{$^{#1}$}\kern0.15mm}% = L^p with p a bit lifted
\def\lll^#1{\ell\kern0.7mm\raise.2mm\hbox{$^{#1}$}\kern0.15mm}% = l^p with p a bit lifted
\def\Cinfty{C\kern.6mm\raise.15mm\hbox{$^\infty$}\kern.07mm} % C^{\infty}
\def\Cinftyzero{C\raise.2mm\hbox{$^{\,\infty}_{\hbox{\font\=cmr5\0}}$}} % C_0^\infty(O)
\def\Cper^#1{C^{\,#1}_{\raise.15mm\hbox{\font\=cmr5\per}}} % use e.g. as \Cper^i(\Re\yi N) , C_{per}^k(O) not= Univ only if O=R^N with N in IN
\def\Cperinfty{C\raise.2mm\hbox{$^{\kern.55mm\infty}_{\raise.15mm\hbox{\font\=cmr5\per}}$\kern.15mm}} % gives C_{per}^\infty
\def\Idc{\roman{{Id_{\kern.3mm}}_{c}\,}} % \Idc\bosycal C , the class of identities in a category Cal C
\def\Catal_#1{{\raise1.3mm\hbox{\font\=cmr5\c}\alpha_{\kern0.3mm}}_{\boldsymbol{\mathcal#1}}\kern.45mm} % \Catal_Cf = ^c\alpha_{\bold cal C} f , the categorical beginning of a morphism f in Cal C
\def\Catom_#1{{\raise1.3mm\hbox{\font\=cmr5\c}\omega_{\kern0.3mm}}_{\boldsymbol{\mathcal#1}}\kern.45mm} % \Catom_Cf = ^c\omega_{\bold cal C} f , the categorical end       of a morphism f in Cal C
\def\baral_#1{\bar\alpha\kern.45mm\lower.2mm\hbox{$_{\boldsymbol{\mathcal#1}}\kern.45mm$}} % \Catal_Cf = ^c\alpha_{\bold cal C} f , the categorical beginning of a morphism f in Cal C
\def\barom_#1{\bar\omega\kern.45mm\lower.2mm\hbox{$_{\boldsymbol{\mathcal#1}}\kern.45mm$}} % \Catom_Cf = ^c\omega_{\bold cal C} f , the categorical end       of a morphism f in Cal C
\def\catimes{\kern.95mm\raise.45mm\hbox{\font\=cmbsy6\\char'002}\kern-2.2mm\lower.7mm\hbox{\font\=cmr5\c\kern-.15mm a}\kern1.05mm} % \bosycal X\catimes\bosycal Y = the product category of bCal X and bCal Y
\def\cstimes{\kern.95mm\raise.45mm\hbox{\font\=cmbsy6\\char'002}\kern-2.1mm\lower.7mm\hbox{\font\=cmr5\c\kern-.15mm s}\kern1.2mm} % \bosycal X\cstimes\bosycal Y = the standard product of standard categories bCal X and bCal Y
\def\opcat{\,\raise1.7mm\hbox{\font\=cmr5\op}} % \bosycal C\opcat = bCal C^{op} = the categorical opposite of bCal C

\def\bold#1{{\bf#1}}
\def\roman#1{{\rm#1}}
\def\limu_#1{\lim\kern-5.5mm\lower1.5mm\hbox{$_{#1}\ $}}
\def\oseoy{\raise1.9mm\hbox{\kern.5mm\font\=cmr5\o}\kern-1.7mm y}% kahdessa kohtaa esiintyy
\def\O{{}^{}\Cal O}
\def\Univ{\hbox{\font\=cmssbx10\U}{}} % the class of all sets
\def\Pows{\Cal P\kern-.7mm\lower.15mm\hbox{$_s$}\kern.3mm} % \Pows A = { U : U subseteq A }
\def\lei{      {}_{ {}^{\,\downarrow\text{\hskip-2.1mm}       }  }  \cap       }
\def\lei{\hbox{\kern.45mm$_{^\downarrow}\kern-1.280mm\cap\kern.85mm$}}
\def\leip{\hbox{\kern.6mm$_{^\downarrow}\kern-1.280mm\cap\kern1mm$}}
\def\leiss#1#2{\hbox{\kern.5mm\kern.#1mm$_{^\downarrow}\kern-1.280mm\cap\kern.9mm\kern.#2mm$}} % \leiss01 = kern.05mm \lei \kern.15mm
\def\vfsupp{\text{\,-\,\raise1.4mm\hbox{\font\=cmr5\f}supp\kern1mm}} % (\Cal T,o)\vfsupp f = (T,o)-supp f = the preceding generalized; o for 0 and dom f inc bigcup Cal T with Cal T a topology

\def\inve{\lower.85mm\hbox{$^{^-}$}\kern-.5mm{}^\iota}

\def\fvalue{\hbox{\kern.2mm\font\=cmr10\\char'022\kern-.2mm}} % f\value x=f(x)
\def\ffvalue{\hbox{\kern.2mm\font\=cmr7\\char'022\kern-.2mm}} % f\value x=f(x)
\def\image{\hbox{\font\=cmr10\\char'022\kern-1mm\char'022}} % f\image A=f[A]
\def\iimage{\hbox{\font\=cmr7\\kern.3mm\char'022\kern-.7mm\char'022\kern-.3mm}} % f\image A=f[A]
\def\images{\hbox{\font\=cmr10\\char'022\kern-1mm\char'022\kern-1mm\char'022}} % f\images\Cal A={f[A]:A in Cal A}
\def\weco{\kern.15mm\hbox{\font\=cmtt10\\char'054}\kern.4mm} % Weihe comma: x\weco y=(x,y)_{Weihe}
\def\cdotn{\kern-.2mm\cdot\kern-.2mm} % same as cdot but with smaller spaces before and after
\def\setminusn{\kern-.2mm\setminus\kern-.2mm} % same as setminus but with smaller spaces before and after
\def\capss#1#2{\kern.1mm\kern.#1mm\cap\kern.1mm\kern.#2mm} % \capss21 = kern.3mm \cap \kern.2mm
\def\cupss#1#2{\kern.1mm\kern.#1mm\cup\kern.1mm\kern.#2mm} % \cupss21 = kern.3mm \cup \kern.2mm
\def\cuppp{\kern0.15mm\cup\kern0.15mm} % \cup with larger spaces around
\def\timesn{\kern-.2mm\times\kern-.2mm} % same as times but with smaller spaces before and after
\def\Times{\kern.7mm\hbox{\font\=cmbsy10\\char'002}\kern.7mm}
\def\ttimes{\hbox{\kern-.2mm${}\times\kern-2.5mm\lower.8mm\hbox{\font\=cmr5\t}\kern1.8mm$}} % product topology = \rist
\def\ttimesn{\hbox{\kern-.2mm${}\times\kern-2.5mm\lower.8mm\hbox{\font\=cmr5\t}\kern1.4mm$}} % same as above but with smaller space after 
\def\ktimes{\hbox{\kern-.2mm${}\times\kern-2.5mm\lower1mm\hbox{\font\=cmr5\k}\kern1.5mm$}} % compactly generated product topology
\def\ltimes{\hbox{${}\times\kern-2.45mm\lower.8mm\hbox{\font\=cmr5\l}\kern1.6mm$}} % \varLambda\ltimes\varGamma = the convergence product of varLambda and varGamma
\def\vstimes{\kern.95mm\raise.45mm\hbox{\font\=cmbsy6\\char'002}\kern-2.3mm\lower.9mm\hbox{\font\=cmr5\vs}\kern1.05mm} % X\vstimes Y = the vector space product of X and Y
\def\atimes{\kern.8mm\hbox{\font\=cmsy10\\char'002}\kern-1.85mm\lower.65mm\hbox{\font\=cmr5\a}\kern1.4mm} % A\atimes B = algebra product of A and B
\def\risti#1{{}^{}\,{\times}_{{\!}_{{}_{\!\!\!{#1}{}^{\!}\ }}}}
\def\Circ{\kern.9mm\hbox{\font\=cmbsy10\\char'016}\kern.9mm}
\def\cardplus{\hbox{$\kern.77mm+\kern-1.95mm\raise.23mm\hbox{$_{_{\roman c}}$}\kern1.33mm$}}%
\def\ordplus{\hbox{$\kern.78mm+\kern-1.97mm\raise.23mm\hbox{$_{_{\roman o}}$}\kern1.22mm$}}%

\def\ccdot{\hbox{$\kern.77mm\cdot\kern-1mm\raise.45mm\hbox{$_{_{\roman c}}$}\kern.38mm$}} % i\ccdot j = the cardinal product of i and j
\def\svs#1{\sbi{\fiveroman{svs\,}#1}} % for example (x+tu)\svs E = (x+tu)_{svs E} = " x+tu " in the ambiguous notation

\def\minus{\kern.2mm\lower1.05mm\hbox{$^-$}}
\def\pplus{\raise.22mm\hbox{\font\=cmr5\\char'053}}% 5 point +
\def\mminus{\raise.18mm\hbox{\font\=cmsy5\\char'000}}% 5 point -
\def\plusinftyy{\raise.18mm\hbox{\font\=cmr5\\char'053}\infty}% +infty for sub- and superscripts with smaller +
\def\minusinftyy{\raise.18mm\hbox{\font\=cmsy5\\char'000}\infty}% -infty for sub- and superscripts with smaller -
\def\inftyy{\raise.15mm\hbox{\font\=cmsy7\\char'061}} % a little raised infty for superscript
\def\inftyyplus{\raise.15mm\hbox{\font\=cmsy7\\char'061}\raise.65mm\hbox{\font\=cmr5\+}} % a little raised infty^+ for superscript
\def\plusinfty{\lower1.05mm\hbox{$^+$}\infty}
\def\minusinfty{\lower1.05mm\hbox{$^-$}\infty}
\def\Qe{\hbox{$Q\kern-2.6mm\raise.2mm\hbox{\font\=cmssqi8\I}\kern1.7mm$}}% the set of rational numbers
\def\Re{{I\!\!R}{}} % the set of real numbers

\def\Ce{{\hbox{$C\kern-2.5mm\raise.2mm\hbox{\font\=cmssqi8\I}\kern1.48mm$}}}
\def\imag{\kern.15mm\lower.6mm\hbox{$^{^*}$}\kern-1.8mm\imath\kern.1mm} % the imaginary unit
\def\ebit#1{\kern.1mm\hbox{\font\=cmmib8\#1}\kern.2mm} % use e.g. as $\ebit A$ to get 8 point bold italic A
\def\ebiF{\kern.1mm\hbox{\font\=cmmib8\F}\kern.5mm} % 8 point bold italic F
\def\ebiT{\kern.1mm\hbox{\font\=cmmib8\T}\kern.6mm} % 8 point bold italic T
\def\ebiU{\kern.1mm\hbox{\font\=cmmib8\U}\kern.5mm} % 8 point bold italic U
\def\bmii#1#2{\hbox{\font\=cmmib#1\#2}}

\def\fssi#1{\hbox{\font\=cmssi10\#1}\kern0.15mm} % cmssi in math
\def\smb#1{\hbox{\font\†=cmmi8\†#1\kern.3mm}} % eightpoint (capital) math symbols
\def\eCal#1{\kern.1mm\hbox{\font\†=cmbsy8\†#1\kern.4mm}} % 8point bold calligraphic #1
\def\ecal#1{\kern.1mm\hbox{\font\†=cmsy8\†#1\kern.3mm}} % 8point calligraphic #1
\def\ncal#1{\kern.1mm\hbox{\font\†=cmsy9\†#1\kern.3mm}} % 9point calligraphic #1
\def\vcal#1{\kern-.1mm\vec{\kern.2mm\hbox{\font\†=cmsy7\†#1}\kern.3mm}} % e.g., vector bundle \vcal E

\def\id{\kern.3mm\roman{id}\kern.7mm}
\def\idv{\hbox{\font\=cmr10\id}\kern.25mm\lower.7mm\hbox{\font\=cmr6\v}\kern.6mm} % \idv E = id_v E = id(v_s E)
\def\idm{\hbox{\font\=cmr10\id}\kern.25mm\lower.7mm\hbox{\font\=cmr6\m}\kern.6mm} % id_m M=id(Dom M)
\def\seq#1{\langle#1\rangle}

\def\SemiNor{\Cal S_{_N}\kern0.15mm}% \SemiNor E = set of continuous seminorms in E
\def\vecs{\upsilon\kern-0.3mm\lower.15mm\hbox{$_s$}\kern0.3mm} % underlying set of a structured vector space
\def\vecss{\hbox{\font\=cmitt10\v}\kern-0.1mm\lower.15mm\hbox{$_s$}\kern0.2mm} % underlying set of a vector space / vector (= linear) struc
\def\nullv{0\lower.7mm\hbox{\font\=cmr6\v}\kern.6mm} % \nullv X = zero vector of the vector space X
\def\nullsv{0\lower.7mm\hbox{\font\=cmr6\sv}\kern.6mm} % \nullsv E = zero element in a structured module E
\def\Bnull_#1{\hbox{\font\=cmssbx10\0}{_{\kern-0.1mm}}_{#1\kern.15mm}} % \Bnull_E = \nullsv E
\def\Bzero_#1{\hbox{\font\=cmbx10\0}{_{\kern-0.1mm}}_{#1}} % \bzero X = zero in vector space X
\def\bnull#1{\hbox{\font\=cmssbx10\0}{}_{\font\=cmmi6\lower.15mm\hbox{\kern-.1mm\#1\kern.15mm}}} % \bnull E = \Bnull_E = \nullsv E
\def\bnulla#1_#2{\hbox{\font\=cmssbx10\0}{}_{\font\=cmmi6\lower.15mm\hbox{\#1\kern-.1mm}}\lower.3mm\hbox{$_{_{#2}}$}} % \bnulla E_1 = zero in vector space E_1
\def\bzero#1{\hbox{\font\=cmbx10\0}{}_{\font\=cmmi6\lower.15mm\hbox{\kern-.1mm\#1\kern.15mm}}} % \bzero X = zero in vector space X
\def\dom{{{}^{}{\rm dom}\,{}_{{}^{}}}}
\def\domm{\kern0.15mm{\rm dom}{^{\kern.3mm\hbox{\font\=cmr6\2}}}\,}
\def\domr#1{\roman{dom}^{\font\=cmr6\raise.0mm\hbox{\kern.3mm\#1}}}

\def\rng{{}^{}{\rm rng}\,{}_{{}^{}}}
\def\CeeFB^#1{C\kern-0.15mm\lower.85mm\hbox{\font\=cmr5\FB}\kern-2.4mm^{#1}\kern.1mm} % \CeeFB^k gives C_{FK}^k
\def\CeePi^#1{C\kern-0.15mm\lower.85mm\hbox{\font\=cmr5\\char'005}\kern-1mm^{#1}\kern.1mm} % \CeePi^k gives C_{\Pi}^k
\def\Ccinfty{C\lower.3mm\hbox{$\kern-0.2mm_{\roman c}$}\kern-.85mm\raise.3mm\hbox{$^\infty$}\kern0.15mm} % gives C_c^{infty}
\def\CPi#1{C\kern-.2mm\lower.05mm\hbox{$_{_\Pi}$}\kern-1.52mm{}^{#1}}
\def\CinftyPi{C\kern.4mm\raise.3mm\hbox{$^\infty$}\kern-3.35mm_{_\Pi}\kern1.45mm}
\def\CinftyS{\Cinfty\kern-3.9mm_{_{\Cal S}}\kern1.45mm}

\def\RHB#1#2{\raise#1mm\hbox{$#2$}} % raised (by #1 mm) horizontal box of #2
\def\LHB#1#2{\lower#1mm\hbox{$#2$}} % lowered (by #1 mm) horizontal box of #2

\def\fiveroman#1{\hbox{\font\=cmr5\#1\kern.1mm}}

\def\sixroman#1{\hbox{\font\=cmr6\#1\kern.1mm}}
\def\eightmath#1{\hbox{\font\=cmmi8\{#1}\kern.1mm}}% 8 point math italic
\def\eightroman#1{\hbox{\font\=cmr8\{#1}\kern.1mm}}% 8 point roman
\def\erm#1{{\font\=cmr8\#1}}% 8 point roman for text, e.g. \erm{DF\,}--\,space 
\def\subtext#1{\raise.2mm\hbox{$_{_{\kern0.15mm\roman{#1}}}$}}
\def\subtexT#1{\raise.2mm\hbox{$_{_{\kern0.15mm\hbox{\font\=cmr5\#1}}}$}}
\def\sNor#1{\kern.25mm\lower.38mm\hbox{$_{#1}$}}
\def\sNorr#1{\kern-.2mm\lower.38mm\hbox{$_{#1}$}}
\def\sNoreset_#1{\kern.13mm\lower.83mm\hbox{\font\=cmmi6\C}\kern.32mm\lower.1mm\hbox{$_{^{\emptyset,#1}}$}}% \|y\|\sNoreset_i produces \|y\|_{C^{\emptyset,i}}
\def\sbi#1{{_{\kern-0.1mm}}_{#1}} % same as _#1 but a little lower
\def\ar#1{{}_{\font\=cmr6\lower.15mm\hbox{\kern.1mm\#1}}} % 6:n pisteen numeroalaindeksi
\def\yplus{\lower1mm\hbox{$^{^+}$}} % + merkki edelliseen
\def\yminus{\lower1mm\hbox{$^{^-}$}} % - merkki edelliseen
\def\aminus{{\kern.15mm\raise.3mm\hbox{$_{_-}$}\kern-.1mm}}%
\def\yvee{\LHB{.9}{^{^{\,\vee}}}\kern-.3mm} % like ^\vee
\def\ywed{\LHB{.9}{^{^{\,\wedge}}}\kern-.3mm} % like ^\wedge
\def\adot{\kern.2mm\hbox{\font\=cmb10\\char'056}}%
\def\ydot{\kern.2mm\raise1.9mm\hbox{\font\=cmb10\\char'056}}% k\ydot = k^. = the integer corresponding to the natural number k
\def\yydot{\kern.2mm\raise1.35mm\hbox{\font\=cmb7\\char'056}\kern.2mm}% the above for 7 point mode
\def\yydott{\kern.2mm\raise1.35mm\hbox{\font\=cmb6\\char'056}\kern.2mm}% the above for 7 point mode
\def\ClT{{\rm Cl}\kern.25mm\lower.4mm\hbox{$_{\Cal T}$}\kern0.2mm} % Cl_Cal T A = closure of A w.r.t. T
\def\IntT{\sp{\rm Int}\kern.2mm\lower.4mm\hbox{$_{\Cal T}$}\kern0.2mm} % Int_Cal T A = interior of A w.r.t. T
\def\Cl_taurd#1{\roman{Cl_{}}_{\kern0.37mm\hbox{\font\=cmmi8\\char'034}\kern-0.15mm{_{}}_{\roman{rd}}\kern0.2mm#1\,}}% Cl_{tau_{rd}#1}S = closure of S in the topology of the topologized vector space E (= #1)
\def\Int_taurd#1{\roman{Int_{}}_{\kern0.37mm\hbox{\font\=cmmi8\\char'034}\kern-0.15mm{_{}}_{\roman{rd}}\kern0.2mm#1\,}}% Int_{tau_{rd}#1}S = interior of S in the topology of the topologized vector space E (= #1)
\def\inc{\subseteq}
\def\iinc{\supseteq}
\def\exi#1{\exists\,#1\kern.2mm\,;}
\def\all#1{\forall\,#1\kern.2mm\,;}
\def\imply{\Rightarrow}
\def\equivv{\Leftrightarrow}

\def\spp{\kern0.07mm} % a horizontal very small positive space
\def\sp{\kern0.15mm} % a horizontal small positive space
\def\ssp{\kern0.37mm} % a bigger positive space
\def\snn{\kern-0.2mm} % a very small horizontal negative space
\def\sn{\kern-0.3mm} % a horizontal small negative space
\def\ssn{\kern-0.63mm} % a bigger negative space
\def\biggerlineskip#1 {\linebreak\nopagebreak\vskip-4.2mm\vskip.#1mm\nopagebreak\noindent}%
\def\Biggerlineskip#1 {\linebreak\nopagebreak\vskip-4.2mm\vskip#1pt\nopagebreak\noindent}%
\def\nKP#1{$\null$\kern#1mm}
\def\KP#1{\kern#1mm} % posive kern of #1 mm
\def\KN#1{\kern-#1mm} % negave kern of #1 mm
\def\nhskip#1mm{$\null$\kern#1mm}

\def\mhyppy#1{\null\kern#1mm}

\def\text#1{\hbox{\rm#1}}
\def\NS{\vskip1.7mm}

\def\VBOX/#1/#2/HEREend{\vbox{#2\vskip-#1mm}\vfill\null\eject}
\def\œ$#1${\hbox{$#1$}} % text math which is not compressed or stretched
\def\"{\"a} \def\"{\"o}
\def\q#1{``\kern0.37mm#1\kern0.37mm"}
\def\noin{\noindent}
\def\Newline{\kern-10mm\newline}

\def\eps{\varepsilon}
\def\leu{\raise1.5mm\hbox{\font\=cmmi5\\char'074}\kern.2mm}%
\def\riu{\kern.2mm\raise1.5mm\hbox{\font\=cmmi5\\char'076}}%
\def\Symbol#1Ï{\kern.35mm\hbox{\font\=cmr10\\char'047}\kern.2mm#1\kern.35mm\hbox{\font\=cmr10\\char'047}}%
\def\Symboo#1Ï{\kern.35mm\text{`}\kern.2mm#1\kern.35mm\hbox{\font\=cmr10\\char'047}}%%

\def\newskline#1 \par{\newline$\null$\kern#1mm}%
\def\vinskip#1#2 \par{\vskip#1mm$\null$\kern-6mm\hspace{#2mm}}
\newenvironment{myLeftskip}[3]{\leftskip#1mm\parindent0mm\addtolength{\parskip}{#2mm}\addtolength{\baselineskip}{#3mm}\def\newfline \par{\newline$\null\hfill$}}{\par}
\def\binsubsubhead#1#2\par{\vskip4mm{\bf#1.}\hskip5mm{\font\=cmss10\#2}\nopagebreak\vskip2mm\nopagebreak\noindent}%
\def\sigrd{\sigma\kern-.2mm\lower.7mm\hbox{\font\=cmr6\r\font\=cmr5\d}\kern.6mm}
\def\ssigrd{\sigma\kern-.2mm\lower.7mm\hbox{\font\=cmr6\r\font\=cmr5\d}\kern-1.7mm\raise1.25mm\hbox{\font\=cmr6\2}\kern1mm}
\def\sssigrd{\sigma\kern-.2mm\lower.7mm\hbox{\font\=cmr6\r\font\=cmr5\d}\kern-1.7mm\raise1.25mm\hbox{\font\=cmr6\3}\kern1mm}
\def\sigrdu^#1{\sigma\kern-.2mm\lower.7mm\hbox{\font\=cmr6\r\font\=cmr5\d}\kern-1.7mm\raise1.25mm\hbox{\font\=cmr6\#1}\kern1mm}
\def\taurd{\tau\kern-.4mm\lower.7mm\hbox{\font\=cmr6\r\font\=cmr5\d}\kern.6mm}
\def\tsigrd{\tau\sigma\kern-.2mm\lower.7mm\hbox{\font\=cmr6\r\font\=cmr5\d}\kern.6mm}
\def\tauRe{\tau{_{\kern-0.6mm}}_{\hbox{\font\=cmmi5\I\!\!R}}} % \tauRe = \tau_{IR} = the natural topology of the real line
\def\bartauRe{\bar\tau{_{\kern-0.6mm}}_{\hbox{\font\=cmmi5\I\!\!R}}} % \bartauRe = \bar\tau_{IR} = the natural topology of the extended real line
\def\tauR#1{\tau_{_{I\!\!R}}\kern-1.5mm^{#1}}
\def\RN{I\!\!R\kern.3mm^{\hbox{\font\=cmmi6\N}}} % Re^N
\def\QTN{Q\kern.1mm_{\lower.2mm\hbox{\font\=cmmi6\T}}^{\kern.2mm\hbox{\font\=cmmi6\N}}} % Q_T^N
\def\lleLCS(#1)-{\hbox{\font\=cmsy10\\char'024}\kern.3mm\lower.62mm\hbox{\font\=cmr5\LCS}\kern.4mm(\kern0.15mm#1\kern0.37mm)\text{\,-\kern0.15mm}} % use e.g. as in $\lleLCS(\snn\bosy K)-\sup\,\Cal H$ or $\lleLCS(\tfbbR)-\sup\,\Cal H$
\def\leLCSr{\hbox{\font\=cmsy10\\char'024}\kern.3mm\lower.62mm\hbox{\font\=cmr5\LCS}\kern.4mm}
\def\leLCS-{{\le}{}_{_{{\rm LCS}}}\text{\sp-\sp}}
\def\tvpreceq{\kern1.4mm\raise1.7mm\hbox{\font\=cmr5\v}\kern-1.2mm\lower.25mm\hbox{\font\=cmr5\t}\kern-1.4mm\hbox{\font\=cmsy10\\char'026}\kern1mm} % E tvpreceq F means the same as the earlier " E le F " for topological vector space E,F
\def\Centerline#1\par#2\par#3{\noindent#1\phantom{#3}\hfill#2\hfill\phantom{#1}#3}

\def\vbDelta_#1{\text{\kern.4mm-\,}\bar\Delta{_{\kern.3mm}}_{#1}\kern.6mm}
\def\CalDBGN^#1{\Cal D_{\hbox{\font\=cmr5\BGN}}^{\kern.6mm{#1}}\kern.15mm}

\def\sacap{\hbox{${}\kern.15mm\sqcap\kern-2.6mm\raise.3mm\hbox{\font\=cmr5\a}\kern1.3mm{}$}} % E\skcap F = product of arc-generated vector spaces
\def\sdcap{\hbox{${}\kern.15mm\sqcap\kern-2.6mm\raise.3mm\hbox{\font\=cmr5\d}\kern1.3mm{}$}} % E\skcap F = product of dualized vector spaces
\def\LipFKa^#1{\mathcal L\kern.15mm ip_{\,\hbox{\font\=cmr5\FK\kern.15mm a}}^{\kern.75mm#1}}
\def\LLip_#1^#2{\mathcal L\kern.15mm ip_{\kern.8mm\vphantom{\hbox{\font\=cmr5\I}}\hbox{\font\=msbm5\#1\,}}^{\vphantom l\kern.9mm#2\kern.3mm}}%
\def\LipFK_#1^#2{\mathcal L\kern.15mm ip_{\,\hbox{\font\=cmr5\FK\kern.15mm#1}}^{\kern.8mm#2}}%
\def\CalCFK0{\mathcal C\kern.2mm\lower.7mm\hbox{\font\=cmr5\FK\kern.3mm\font\=cmr6\0}{}}

\def\avarFK{\raise1.7mm\hbox{\font\=cmr5\a}\kern-.2mm\hbox{\font\=cmtex11\\char'012}\subtext{FK}\kern.4mm} % arc-FK variation
\def\abThetaFK{\raise1.9mm\hbox{\font\=cmr5\a}\kern.7mm\bar{\kern-.9mm\hbox{\font\=cmssi10\\char'002}}\kern.4mm\subtext{FK\,}} % arc-FK extended diff. quo.
\def\abDeltaFK{\raise1.7mm\hbox{\font\=cmr5\a\!}\bar\Delta\kern.4mm\subtext{FK\,}}
\def\svs#1{\sbi{\fiveroman{svs\,}#1}}

\def\AVS{\roman{AVS}\kern.4mm}
\def\DVS{\roman{DVS}\kern.4mm}
\def\strVS{\roman{strVS}\kern.4mm}
\def\Con_#1#2{\roman{Con}\sp\subtext{#1\sp#2}}
\def\ConFKd{\roman{Con}\sp\subtext{FK\sp d}}
\def\ConFKa{\roman{Con}\sp\subtext{FK\sp a}}
\def\ConFKt{\roman{Con}\sp\subtext{FK\sp t}}
\def\ConKM{\roman{Con}\sp\subtext{KM}}

\def\tauMac{\tau\kern.15mm\lower.7mm\hbox{\font\=cmr5\Mac}\kern.6mm}

\def\sbii#1{{_{\kern-0.1mm}}_{\hbox{\font\=cmmi5\#1}}}

\def\deltaAV{\delta{_{\kern-0.1mm}}_{\hbox{\font\=cmr6\av}}} % arc gen to dual
\def\tauDV{\tau{_{\kern-0.4mm}}_{\hbox{\font\=cmr5\d\font\=cmr6\v}}} % dual to arc gen
\def\tauTV{\tau{_{\kern-0.4mm}}_{\hbox{\font\=cmr6\tv}}} % locally convex to arc gen

\begin{document}

\title[$\text{\sc FK differentiabilities as BGN}$]%
      {The Fr\"licher\,--\,Kriegl differentiabilities as a particular\vskip1mm
         case of the Bertram\,--\,Gl\"ckner\,--\,Neeb construction}

\author[S. Hiltunen]{Seppo\ I\. Hiltunen}
\address{Helsinki University of Technology                             \vskip0mm$\hspace{2mm}$
           Institute of Mathematics, U311                              \vskip0mm$\hspace{2mm}$
           P.O.\ Box 1100                                              \vskip0mm$\hspace{2mm}$
           FIN-02015 HUT\vskip0mm
         FINLAND}
\email{shiltune\,@\,cc.hut.fi}



\begin{abstract}

We prove that the order $k$ differentiability classes for \œ$k=0.\ssp,1.\ssp,
\ldots\,\infty$ in the \q{arc-generated} interpretation of the Lipschitz
theory of differentiation by Fr\"licher and Kriegl can be obtained as
particular cases of the general construction by Bertram, Gl\"ckner and Neeb
leading to $C^{\ssp k}$ differentiabilities from a given $C^{\,0}$ concept.

  \end{abstract}

\maketitle

% ----------------------------------------------------------------------------

\noin In \cite[p.\ 252]{BGN}\ssp, it was already announced (without proof\,)
that the general construction in \cite{BGN} gives as particular cases the
Lipschitz differentiabilities of order $\ssp k$ in \cite{FK} for maps \math{
f:E\iinc U\to F} when the spaces \math{E\ssp,F} are equipped with the $\,
\roman c\,^{\infty\,}$--\,extensions of the locally convex topologies. Our
purpose in this note is to prove this result, formulated as
Theorem \ref{main Thm} below. This complements our treatment in
\cite{SeBGN}\ssp. For most of the notations and of the preliminaries, we refer
to \cite[pp.\ 4\,--\,9]{SeBGN} which we assume the reader to be acquainted
with. However, we introduce the following slight change in notation.

Below, the standard topological field of real numbers is \math{ \tfbbR =
(\sp\fbbR\sp\,,\taubb_R)} with underlying set \math{\bbR} instead of the
earlier \math{\bosy R=(\ssp\bold R\sp\,,\tauRe)} over \math{\Re\ssp}. We also
let \œ$\ssp\rbb R^+=$ $\bbR\capss20\{\,t:t>0\,\}\,$, and \math{\bbNo=\infty}
is the set of natural numbers. Further, we define \vskip.5mm $\nKP{18}
[\,\sp f\sp,\sp g\rbrakf = \{\ssp(\sp x\,;y\ssp,z\sp) : (\sp x\ssp,y\sp)\in f\ssp$
                           and $\ssp(\sp x\ssp,z\sp)\in g\,\} \nKP{19} $ and \vskip.3mm $\nKP{20.8}
f\ftimes g = \{\ssp(\sp x\ssp,u\,;y\ssp,v\sp) : (\sp x\ssp,y\sp)\in f\ssp$
                           and $\ssp(\sp u\ssp,v\sp)\in g\,\}$\vskip.5mm

\noin which earlier were written \q{\sp$[\,\sp f\sp,\sp g\ssp\,]$\sp} and \q{\snn$
  f\risti2 g$\ssp}, respectively. \vskip.3mm

In general definitions, to fix matters precisely, we utilize the notational
convention for linear combinations sketched in \cite[p.\ 5]{SeBGN} according
to which for example we have \math{ (\sp t\, x + s\, y\sp)\svs E =
 \ssigrd E\ssp\fvalue(\sp\tsigrd E\ssp\fvalue(\sp t\ssp, x\sp)\ssp,
  \tsigrd E\ssp\fvalue(\sp s\ssp, y\sp)) } which more shortly and generally
ambiguously is denoted by \q{\ssp\œ$ t\, x + s\, y $\ssp} when \math{E} is
a real structured vector space and we have \math{s\ssp,t\in\bbR} and \math{
x\ssp,y\in\vecs E\sp}. Likewise, for example instead of the precise \math{
[\,\sp\{\sp x\sp\}+\{\ssp\delta\ssp\}\sp\,B\ssp\,]\svs E} one conventionally
writes \q{\ssp\œ$x+\delta\,B$\ssp}. However, in passages of informal
discussion or in proofs where the surrounding spaces have been fixed, we may
use the shorter imprecise notations.

By a {\it structure changer\ssp} meaning any function \œ$\ssp
\sigma\inc(\sp\Univ\sp^{\times 2.})^{\times 2.} =
\Univ\times\Univ\times\snn(\sp\Univ\times\Univ\sp)$ with \œ$\ssp
\roman{pr}\ar 1\snn\circ\sigma\inc\roman{pr}\ar 1\ssp$,
the {\it interpretation\ssp} of a class $\ssp\Cal C\sp$ of $\,\bold K\,
$--\,vector maps by a struct- ure changer $\sp\sigma\sp$ satisfying also \œ$\ssp
\domm\Cal C\sp\cup\sp\rng\snn\dom\Cal C\inc\dom\sigma\sp$ we understand the \linebreak
class $\,\sigma\ftimes\sigma\ftimes\id\sn\image\Cal C =
\{\ssp(\sp\sigma\spp\fvalue E\ssp,\sigma\spp\fvalue F\sp,f\ssp) :
             (E\ssp,F\sp,f\ssp)\in\Cal C\sp\,\}\,$. In \cite{FK} structure
changers frequently appear as object components of functors.      \vskip.3mm

We next suitably reformulate the facts from \cite{FK} which we need below.

% ----------------------------------------------------------------------------

  \binsubsubhead A{The convenient spaces of Fr\"licher and Kriegl}

Let \math{ \strVS(\sp\fbbR\ssp)
  = \{\ssp(X\sp,S\ssp):X\ssp$ real vector space$\sp\,\}\sp}, the class of all
structured real vector spaces. Its subclass \math{\LCS(\sp\tfbbR\ssp) } has
only a minor role below. The subclasses $\roman{AVS}\ssp$ and \math{\roman{DVS}}
will be more important here. They are obtained as follows.

Put \math{\DVS=\strVS(\sp\fbbR\ssp)\capss10\{\,E:(E\ssp)\subtext D\ssp\}}
where \math{(E\ssp)\subtext D} means that \math{\taurd E\not=\emptyset} is a
point separating linearly closed set of linear maps \math{\sigrd E\to\fbbR\ssp
}. Calling a topology $\Cal T\sp$ {\it arc\,-\ssp generated\,} if{}f \math{ U
\in \Cal T} for every \math{U\inc\bigcup\ssp\Cal T} with \math{ c\inve\image U
\in \taubb_R} for all continuous \œ$c:\taubb_R\to\Cal T\ssp$ having \math{
\dom c = \bbR\ssp}, we define \math{ \AVS =
     \strVS(\sp\fbbR\ssp)\capss10\{\,E:(E\ssp)\subtext A\ssp\}} where $
(E\ssp)\subtext A\sp$ means that \math{\taurd E} is an arc\,-\ssp generated
Hausdorff topology for \math{\vecs E} such that for \math{ \Cal P =
(\sp\taubb_R\sp,\taurd E\ssp)} we have \math{
(\ssp\Cal P\spp,\ssigrd E\circ[\,\sp c\ar 1\sp,c\ar 2\sn\rbrakf)} and \math{
(\ssp\Cal P\spp,\tsigrd E\circ[\,\sp c\,,c\ar 1\sn\rbrakf)} topological maps
when \math{(\sp\taubb_R\sp,\taubb_R\sp,c\,)} and \math{
           (\ssp\Cal P\spp,c\sp\sbi\iota)} are global topological maps for $
  \sbi{\iota\ssp=\ssp\sixroman{1\sp,\ssp 2}\,}$.

The reader may compare the preceding definitions to \cite[Definition 2.1.1,
Remark, p.\ 28, Definitions 2.3.8, 2.3.9, p.\ 45]{FK}\ssp. In particular, note
that we make all the spaces \q{separated} right from the beginning, as opposed
to \cite{FK}\ssp. By a simple verification \math{E\in\AVS} when \math{ E \in
\TVS(\sp\tfbbR\ssp)} with \math{\taurd E} a metrizable topology. \vskip.4mm

We say that \math{\Cal T} is the {\it smooth $\ssp S\,$--\,topology\ssp} if{}f
there is \math{\Omega} with \math{\emptyset\not=S\inc\ssbb02 R\,^\Omega} and \vskip.3mm $\nKP{7}
\Cal T = \{\,\sp U\sn:U\inc\Omega\ssp$ and $\,
              \all{\sp c\sp}\,c\in\Omega\,^{\ssbb21 R}$ and \vskip.15mm $\nKP{21}
[\ \all{\varphi\in S}\,\varphi\circ c\in\vecs\Cinfty(\ssbb22 R)\ ]\ssp
    \imply\ssp c\inve\image U\in\taubb_R\ssp\}\,$. \hfill Then putting \vskip.3mm $\mhyppy{5.3}
\deltaAV = \seq{\,(\sp\sigrd E\ssp,\spp\Cal L\,(\spp E\ssp,\sn\tfbbR\ssp)) :
                             E \in \AVS\,}\KP{42.3}$ and \vskip.3mm \centerline{$
\tauDV = \seq{\,(\sp\sigrd E\ssp,\Cal T\,):E\in\DVS\ssp$ and $\,
          \Cal T\sp$ is the smooth $\ssp\taurd E\,$--\,topology$\ssp\,}\,$,} \vskip.3mm

\noin we get the structure changers \œ$\,\tauDV:\DVS\to\AVS\,$ and \œ$\,
\deltaAV:\AVS\to\strVS(\sp\tfbbR\ssp)\,$. Here $\,
\tauDV=\tau\subtext M\ssp|\,\sp\DVS\sp$ when $\ssp\tau\sn\subtext M$ is the
object component of the functor \vskip.3mm \centerline{$
\bosy\tau\sn\subtext M:\sp\underline{\sn\roman{DVS}\sn}\ssp\subtext{there}\to
                       \sp\underline{\sn\roman{ArcVS}\sn}\ssp\subtext{there}\,
$ in \,\cite[Definition 2.3.14, p.\ 46]{FK}\ssp.} \vskip.3mm

\noin Taking for example \math{E=L^{\ssp\frac 12}\ssp(\ssp[\,0\,,1\,]\ssp)\sp
}, we have \math{ \taurd\spp(\sp\deltaAV\KN1\fvalue E\ssp) =
\{\,\vecs E\times\snn\{0\}\sp\}\sp}, hence \linebreak $ \deltaAV\KN1\fvalue E
\not\in\DVS\ssp$, and consequently \math{\rng\deltaAV\not\inc\DVS}.   \vskip.3mm

Now, the class of {\it dualized Fr\"licher\,--\,Kriegl convenient\ssp} spaces
is \vskip.3mm \centerline{$
\ConFKd = \DVS\capss01\{\,E:\Cal L\,(\sp\tauDV\KN1\fvalue E\ssp,\sn\tfbbR\ssp)
               \inc\taurd E\ssp$ and $\ssp(E\ssp)\subtext C\ssp\}$}\vskip.3mm

\noin where \math{(E\ssp)\subtext C} means that for every \math{ c \in
(\sp\vecs E\ssp)\,^{\ssbb10 R}} there is some \math{ c\ar 1 \in
(\sp\vecs E\ssp)\,^{\ssbb10 R}} such that if \math{ \ell\circ c \in \vecs
\Cinfty(\ssbb22 R)} holds for all \math{\ell\in\taurd E\sp}, then we also have
\œ$\,\ell\circ c\ar 1=(\ssp\ell\circ c\ssp)\ssp'$ for \math{\ell\in\taurd E\sp
}. The reader may compare this to
\cite[Definitions 2.4.2, 2.5.3, 2.6.3, pp.\ 48, 53, 57]{FK}\ssp. Putting \œ$\,
\ConFKa = \tauDV\KN1\image\ConFKd\,$, then $\ssp\ConFKa$ is the {\it arc\,-\ssp
generated counterpart\ssp} of the class of Fr\"licher\,--\,Kriegl convenient
spaces, and essentially from \cite[Theorems 2.4.3\ssp(vi)\ssp, 2.5.2\ssp,
  pp.\ 49, 53]{FK} we obtain the following

\begin{lemma}\label{FKd eqv FKa}

$\hfill \tauDV\ssp|\,\ConFKd$ is bijective $\,\ConFKd\to\ConFKa \hfill$ \vskip.5mm\nKP{26.05}
and $\ (\sp\tauDV\ssp|\,\ConFKd)\inve = \sp\deltaAV\ssp|\,\ConFKa\,$.

  \end{lemma}

Letting \math{\ConKM=\{\,E:(E\ssp)\subtext{con\ssp KM}\ssp\}\sp}, where \math{
(E\ssp)\subtext{con\ssp KM}} means that \math{E\in\LCS(\sp\tfbbR\ssp)} and is
locally\sp/\,Mackey complete in the sense \cite[p.\ 196]{Jr} or
\cite[Lemma 2.2, p.\ 15]{KM}\ssp, then \math{\ConKM} is the class of spaces
which are \q{convenient} in the sense \cite[The- orem 2.14, p.\ 20]{KM}\ssp.
The class \math{\ConFKt=\{\,E:(E\ssp)\subtext{con\ssp KM}$ and $\ssp E\sp$ is
bornological$\ssp\,\} } is the {\it locally convex counterpart\ssp} of the
class of Fr\"licher\,--\,Kriegl convenient spaces.

For \math{E\in\LCS(\sp\tfbbR\ssp)\sp}, let \math{ \tauMac E = \{\,\sp U\sn : U
$ mopen in $\sp E\sp\,\}} where \math{U} being {\it mopen\ssp} in \linebreak $
E\ssp$ means that \math{U\inc\vecs E} and that for \math{x\in U} and \math{ B
\in\rajou E} there is some \œ$\ssp\delta\in\rbb R^+$ with \œ$\,
[\,\sp\{\sp x\sp\}+\{\ssp\delta\ssp\}\sp\,B\ssp\,]\svs E\inc U\spp$. Then \math{
\tauMac E} is the {\it Mackey closure\ssp} topology of the {\it bornological\ssp}
vector space \math{(\sp\sigrd E\ssp,\rajou E\ssp)} which in
\cite[Definition 2.12, p.\ 19]{KM} is also called the $\,\roman c\,^{\infty\,}${\it
--\,topology\ssp} of the {\it locally convex\ssp} space \math{E\sp}.

Considering a fixed \math{E\in\Con_{FK}d\sp}, when one a bit loosely speaks of
the \q{Mackey closure topology} in various places in \cite{FK} without
explicitly specifying the bornological vector space, which there would be
denoted by \q{\œ$\sp\sigma\ar b\ssp E\ssp$}, note that by
\cite[Proposition 2.3.7, p.\ 44]{FK} one then refers precily to the topology \math{
\taurd(\sp\tauDV\KN1\fvalue E\ssp)\sp}. This should be kept in mind when we
below refer to results in \cite{FK} in order to get some shortening of
  presentation.

\NS To summarize, if we put \math{ \alpha\sp\subtext{FK\sp t} =
\seq{\,(\sp\sigrd E\ssp,\tauMac E\ssp) : E \in \Con_{FK}t\ssp} } and \œ$\,
\delta\subtext{FK\sp a} = $ $ \deltaAV\ssp|\,\ConFKa\,$, \,we have \vspace{-.5mm}\par$\KP{44.5}
\alpha\sp\subtext{FK\sp t}$ \vspace{-1.5mm}\par\centerline{$
\AVS\supset\Con_{FK}a
                      \,\longleftarrow\KN{1.8}\frac{\KP{5}}{}\ \,
           \Con_{FK}t\subset\Con_{KM}{}\subset\LCS(\sp\tfbbR\ssp) $} \vspace{1mm}\par$\KP{34}
                      \raise1.7mm\hbox{\font\=cmsy12\\char'152}\kern-1.6mm
                       \lower2mm\hbox{\font\=cmsy12\\char'043}\
 \delta\subtext{FK\sp a} $ \vspace{1mm}\par$\KP{16.7}
\DVS\supset\Con_{FK}d$ \vskip1.5mm

\noin where \math{\alpha\sp\subtext{FK\sp t}} and \math{\delta\subtext{FK\sp a}}
are bijective. This diagram is here given only for the purpose of clarifying
the relations between the various classes of structured vector spaces. Below,
we only need to consider the bijection given by the vertical arrow.
                                        
\begin{definitions}[\ssp product structures\sp]\label{def prods} $\null$ \vskip.7mm

$X\vstimes Y=\bigcap\ssp\{\ssp(\sp a\ssp,c\sp):\exi{a\ar 1\ssp,\sp a\ar 2\ssp,
  c\sp\ar 1\ssp,\sp c\sp\ar 2\in\Univ}\, X = (\sp a\ar 1\sp,c\sp\ar 1)$ and $\,
    Y = (\sp a\ar 2\ssp,c\sp\ar 2)$ \vskip.3mm \nKP{14.5} and $\,
a = \{\ssp(\sp x\ssp,y\,;u\ssp,v\,;z\ssp,w\sp):(\sp x\ssp,u\ssp,z\sp) \in
      a\ar 1$ and $(\sp y\ssp,v\ssp,w\sp)\in a\ar 2\ssp\}$ \vskip.3mm \nKP{14.5} and \ \ \ $\,
c = \{\ssp(\sp t\,;x\ssp,y\,;u\ssp,v\sp):(\sp t\ssp,x\ssp,u\sp)\in c\sp\ar 1$
    and $(\sp t\ssp,y\ssp,v\sp)\in c\sp\ar 2\ssp\}\sp\}\,$, \vskip.5mm

$\Cal T\ttimes\Cal U = \{\,\bigcup\ssp\Cal A : \Cal A\inc \{\,U\times V :
                           U\in\Cal T$ and $\ssp V\in\Cal U\sp\,\}\sp\}\,$, \vskip.5mm

$E\sp\sqcap F = ((\sp\sigrd E\ssp)\vstimes(\sp\sigrd F\ssp)\ssp,
                 (\sp\taurd E\ssp)\ttimes(\sp\taurd F\ssp))\,$, \vskip.5mm

$E\sdcap F = \bigcap\ssp\{\ssp(X\sp,S\ssp) : X =
   (\sp\sigrd E\ssp)\vstimes(\sp\sigrd F\ssp)$ and $\sp
   S = \{\,\ell:\exists\,\ell\ar 1\in\taurd E\ssp,$ \vskip.3mm $\null\hfill
   \ell\ar 2\in\taurd F\,;\,
   \ell = \{\ssp(\sp x\ssp,y\ssp,r+s\sp):(\sp x\ssp,r\sp)\in\ell\ar 1$ and $
   (\sp y\ssp,s\sp)\in\ell\ar 2\ssp\}\sp\}\sp\}\,$, \KP4 \vskip.5mm

$E\sacap F=\tauDV\KN1\fvalue((\sp\deltaAV\KN1\fvalue E\ssp)\sdcap
                             (\sp\deltaAV\KN1\fvalue F\ssp))\,$.  \vskip.7mm

\noin We may call $\sp X\vstimes Y\sp$ the {\it vector space\ssp} (or module) product of
$\sp X$ and $\ssp Y\sp$. The class $\ssp\Cal T\ttimes\Cal U\ssp$ is the
Tihonov topological product of $\ssp\Cal T$ and $\,\Cal U\ssp$.

  \end{definitions}

By \cite[Proposition 3.3.1, p.\ 65]{FK}\ssp, when \math{E\ssp,F\in\ConFKd\sp},
we have that \œ$\ssp E\sdcap F $ \linebreak is precisely the product space
which in \cite{FK} is denoted by \q{\snn\œ$E\ \RHB{.55}{_\Pi}\sp\,F\sp$}. In
particular, recalling Lemma \ref{FKd eqv FKa} above, we have that \math{
E\spp,F \in \Con_{FK}d\impss22 E\sdcap F\in\Con_{FK}d } and that \math{
E\spp,F \in \Con_{FK}a\impss22 E\sacap F\in\Con_{FK}a \sp}. By
\cite[Remark 3.3.4, p.\ 67]{FK}\ssp, we have \œ$ E\sacap F = E\sp\sqcap F\ssp$
whenever the spaces \math{ E\ssp,F \in \ConFKa } are such that at least one of
them is finite\ssp-\sp dimensional.

% ----------------------------------------------------------------------------

  \binsubsubhead B{The Lipschitz differentiable maps}

We refer to (\ref{defi Lip_R^kE}) and (\ref{defi LipFKd^k}) and
             (\ref{defi LipFKa^k}) of Constructions \ref{defi Lip ...} below.
In \cite[p.\ 83]{FK}\ssp, it is agreed that \math{f:E\iinc U\to F} is \math{
\LLip_{}^k} if{}f \math{E\ssp,F\in\ConFKd} and \math{ U \in
\taurd\spp(\sp\tauDV\KN1\fvalue E\ssp)} and \œ$f\in(\sp\vecs F\ssp)\,^U\sp$
with \math{f\snn\circ\spp c\in\LLip_R^kF} for \math{c\in\LLip_R^kE} having \œ$\ssp
\rng\sp c\inc U\spp$. Assuming that \œ$\ssp k\in\infty\ssp\yplus\sp$, and also
taking \cite[Definition 1.4.1, Proposition 2.3.7, Lemma 4.3.1,
             pp.\ 22, 44, 99]{FK} into account, it is seen that \math{
f:E\iinc U\to F} is \math{\LLip_{}^k} if and only if \œ$
(\spp E\ssp,\spp F\sp,\sp f\ssp)\in\LipFK_d^k\ssp$ with \œ$\ssp\dom\sn f=U\spp
$. Hence, for \math{k\in\infty\ssp\yplus} we may say that $\ssp\LipFK_d^k$ is
the class of maps which are Lipschitz differentiable of order \math{k} in the
sense of Fr\"licher and Kriegl, and that \math{\LipFK_a^k} is its
   \q{arc\ssp-\sp generated} interpretation.

\begin{constructions}[\sp classes of Lipschitz functions and maps\sp]\label{defi Lip ...}

For all \math{E\ssp,\sp k}, with the restriction \math{E\in\DVS} in
(\ref{defi L4}) and (\ref{defi Lip_R^kE}) below, we let

\begin{enumerate}\begin{myLeftskip}{-3}{.2}{.4}

\item \ $\LLip_{}^{}\ssn =
            \{\,\chi:\eexi{\Omega\in\taubb_R}\,\chi \in \ssbb01 R\,^\Omega\ssp$
            and $\,\aall{t\ar 0\in\Omega}\,\eexi{\delta\in\rbb R^+}$ \newfline

      $\aall{s\ssp,t\in\Omega}\,|\,s-t\ar 0\ssp|+|\,\sp t-t\ar 0\ssp|<\delta
      \impss22 |\,\sp\chi\fvalue s-\chi\fvalue t\ssp\,| \le
      \delta^{\sp\mminus 1\,}|\,s-t\sp\,|\sp\,\}\,$, \KP4 \label{defi L1}

\item \ $\LLip_{}^k = \LLip_{}^{}\ssn\cap\sp\big\{\,\chi:k\in\infty\ssp\yplus\sp$
         and $\,\eexi{\Omega\,,\bosy\chi}\,(\sp\emptyset\,,\sp\chi\sp) \in
          \bosy\chi \in\big (\ssbb30 R\,^\Omega\sp\big)\,^{k\ssp+\sp 1.}$ \newfline

        and $\ssp\aall{i\in k}\,\chi\fvalue i\ssp\yplus =
        (\sp\chi\fvalue i\ssp)\ssp'\in\LLip_{}^{}\big\}\,$, \KP4 \label{defi L2}

\item \ $\LLip_R^k = \LLip_{}^k\snn\cap\ssp\bbR\,^{\ssbb19 R}$, \label{defi L3}

\item \ $\LLip_{}^kE = \{\,c : c \in (\sp\vecs E\ssp)\,^{\roman{dom}\,\spp
         \eightmath c}\sp$ and $\,\all{\ell\in\taurd E}\,
          \ell\sp\circ\sp c\in\LLip_{}^k\sp\}\,$, \label{defi L4}

\item \ $\LLip_R^kE=\LLip_{}^kE\capss22(\sp\vecs E\ssp)\,^{\ssbb19 R}$, \label{defi Lip_R^kE}

\item \ $\LipFK_d^k = \ConFKd\KN{.7}^{\times 2.}\snn\times\Univ \capss20
          \{\ssp(\spp E\ssp,\spp F\sp,\sp f\ssp) : k\in\infty\ssp\yplus\sp$
            and $\ssp f\in(\sp\vecs F\ssp)\,^{\roman{dom}\,f}$ \newfline

       and $\ssp\dom\sn f\inc\vecs E\ssp$ and $\,\aall{c\in\LLip_{}^kE}\,
       f\snn\circ\sp c\in\LLip_{}^kF\sp\,\}\,$, \KP4 \label{defi LipFKd^k}

\item \ $\LipFK_a^k=\sp\tauDV\ftimes\tauDV\ftimes\id\image\LipFK_d^k\KP1$. \label{defi LipFKa^k}

  \end{myLeftskip}\end{enumerate}\end{constructions}

\begin{lemma}\label{Lip^0 in EF}

Let $\,E\ssp,F\in\Con_{FK}d\ssp${\rm, }and with $\,G=E\sdcap F\sp${\rm, }also
let $\,c\in(\sp\vecs G\ssp)\,^{\roman{dom}\,\spp\eightmath c}\sp$. Then $\,
c\in\LLip_{}^\emptyset G$ if and only if $\,
  \roman{pr}\ar 1\snn\circ\spp c\in\LLip_{}^\emptyset E$ and $\,
  \roman{pr}\ar 2\circ\spp c\in\LLip_{}^\emptyset F\sp$.         \end{lemma}

\begin{proof} Put \math{c\ar 1=\roman{pr}\ar 1\snn\circ\spp c} and \math{
c\ar 2 = \roman{pr}\ar 2\circ\spp c\ssp}. First letting \œ$\ssp c \in
\LLip_{}^\emptyset G\sp$, to have \œ$c\ar 1\in\LLip_{}^\emptyset E\,$, for
arbitrarily fixed \math{\ell\ar 1\in\taurd E} and \math{ \varphi\ar 1 =
\ell\ar 1\snn\circ\sp c\ar 1}, it suffices that \linebreak \œ$\varphi\ar 1 \in
\LLip_{}^\emptyset=\LLip_{}^{}\ssn$. Letting \math{ \ell\ar 2 =
\vecs F\times\snn\{0\} } and \math{ \ell =
\sigrd\ssn\fbbR\circ(\ssp\ell\ar 1\ftimes\ssp\ell\ar 2)} and \œ$\ssp \varphi =$ \linebreak
\œ$\ell\sp\circ\sp c\,$, as \math{\taurd F} is (easily seen to be) a vector
subspace in \œ$\ssp\fbbR\expnota^\upsilon_s\sp F]_{vs}\sp$, we have \œ$\ssp
\ell\ar 2 \in $ \linebreak $\taurd F\sp$, hence $\ssp \ell \in \taurd G\sp$,
whence by \math{c \in \LLip_{}^\emptyset G} further $\ssp\varphi\in\LLip_{}^{}\ssn
$. For all \math{t} having \vskip.3mm \centerline{$
\varphi\ar 1\KN1\fvalue t = \ell\ar 1\snn\circ\sp c\ar 1\ssn\fvalue\sp t + 0 =
\sigrd\ssn\fbbR\circ(\ssp\ell\ar 1\ftimes\ssp\ell\ar 2)\circ
  [\,\sp c\ar 1\sp,\sp c\ar 2\sn\rbrakf\ssn\fvalue\sp t =
\ell\sp\circ\sp c\sp\fvalue\sp t = \varphi\fvalue\sp t\ssp$,} \vskip.3mm

\noin we get $\ssp\varphi\ar 1=\varphi\in\LLip_{}^{}\ssn$, as we wished.
      Similarly, one obtains $\ssp c\ar 2\in\LLip_{}^\emptyset F\sp$.

Conversely, letting \math{c\ar 1\in\LLip_{}^\emptyset E} and \œ$\, c\ar 2 \in
\LLip_{}^\emptyset F\sp$, in order to get \œ$\ssp c \in \LLip_{}^\emptyset G\sp
$, for arbitrarily fixed \math{\ell\ar 1\in\taurd E} and \œ$\, \ell\ar 2 \in
\taurd F\sp$, and for \math{ \ell =
\sigrd\ssn\fbbR\circ(\ssp\ell\ar 1\ftimes\ssp\ell\ar 2)} and \œ$ \varphi =
\ell\sp\circ\sp c\,$, it suffices that \œ$\ssp \varphi \in \LLip_{}^{}\ssn$.
Putting \math{\varphi\sbi\iota=\ell\sbi\iota\snn\circ\sp c\sp\sbi\iota}, then
\œ$\ssp\varphi\ar 1\ssp,\sp\varphi\ar 2\in\LLip_{}^{}\ssn$, and for all \math{
t} we have \vskip.5mm \centerline{$
\varphi\fvalue t = \ell\sp\circ\sp c\sp\fvalue\sp t =
\sigrd\ssn\fbbR\circ(\ssp\ell\ar 1\ftimes\ssp\ell\ar 2)\circ
       [\,\sp c\ar 1\sp,\sp c\ar 2\sn\rbrakf\ssn\fvalue\sp t =
\sigrd\ssn\fbbR\ssp\fvalue(\ssp\ell\ar 1\snn\circ\sp c\ar 1\ssn\fvalue\sp t\,,
                           \ssp\ell\ar 2\circ\sp c\ar 2\ssn\fvalue\sp t\ssp)$} \vskip.3mm $\KP{9.25}
= \sigrd\ssn\fbbR\ssp\fvalue(\sp\varphi\ar 1\KN1\fvalue t\ssp,
                                \varphi\ar 2\KN1\fvalue t\ssp)
= \varphi\ar 1\KN1\fvalue t+(\sp\varphi\ar 2\KN1\fvalue t\ssp)
= (\sp\varphi\ar 1\snn+\varphi\ar 2)\fvalue t\ssp$, \vskip.5mm

\noin and hence \math{\varphi=\varphi\ar 1\snn+\varphi\ar 2\sp}. Noting that \œ$\ssp
\{\,u+v:u\ssp,v\in\LLip_{}^{}\sn\}\inc\LLip_{}^{}\ssn$, we immediately obtain
$\ssp\varphi\in\LLip_{}^{}\ssn$, \,as it was required.           \end{proof}

\begin{corollary}\label{[f,g] is Lip^0}

For all $\,E\ssp,\spp F\sp,\spp G\sp,\sp f\sp,\sp g\,${\rm, }it holds that \vskip.3mm\centerline{$
(\spp E\ssp,\spp F\sp,\sp f\ssp)\ssp,(\spp E\ssp,\spp G\sp,\sp g\sp) \in
  \LipFKa^\emptyset \impss22 (\spp E\ssp,\spp F\sacap G\sp,\sp
    [\,\sp f\sp,\sp g\rbrakf)\in\LipFKa^\emptyset\,$.}       \end{corollary}

\begin{proof} In view of Lemma \ref{FKd eqv FKa}\sp,
Definitions \ref{def prods} and (\ref{defi LipFKa^k}) and (\ref{defi LipFKd^k})
of Constructions \ref{defi Lip ...}\ssp, for arbitrarily given \math{
(\spp E\ssp,\spp F\sp,\sp f\ssp)\ssp,(\spp E\ssp,\spp G\sp,\sp g\sp) \in
  \LipFK_d^\emptyset\sp}, putting \math{ H=F\sdcap G} and \œ$\ssp h = $ $
[\,\sp f\sp,\sp g\rbrakf\,$, for an arbitrarily fixed \math{ c \in
\LLip_{}^\emptyset E} we must show that \œ$\ssp h \circ c \in
\LLip_{}^\emptyset H\sp$. Further putting \math{c\ar 1= f\snn\circ\sp c} and \math{
c\ar 2 = g\sp\circ\sp c\sp}, by definition we have \math{ c\ar 1 \in
\LLip_{}^\emptyset F } and \œ$c\ar 2 \in\LLip_{}^\emptyset G \sp$. By
Lemma \ref{Lip^0 in EF} for \math{ h\circ c\in\LLip_{}^\emptyset H } it
suffices that \œ$\, \roman{pr}\ar 1\snn\circ h\spp\circ\spp c \in
\LLip_{}^\emptyset F$ and \math{ \roman{pr}\ar 2\circ h\spp\circ\spp c \in
\LLip_{}^\emptyset G } hold. This indeed is the case since we have \œ$\,
\roman{pr}\ar 1\snn\circ h\spp\circ\spp c = $ \œ$
\roman{pr}\ar 1\snn\circ[\,\sp f\sp,\sp g\rbrakf\circ\sp c =
\roman{pr}\ar 1\snn\circ[\KP1 f\snn\circ\sp c\,,\sp g\spp\circ\sp c\rbrakf =
\roman{pr}\ar 1\snn\circ[\KP1 c\ar 1\ssp,\sp c\ar 2\rbrakf =
c\ar 1\ssp|\,\dom c\ar 2 \,$, and similarly also \math{ \roman{pr}\ar 2\circ h
\spp\circ\spp c = c\ar 2\ssp|\,\dom c\ar 1}.                     \end{proof}

\begin{proposition}\label{Lip0 is BGN}

\hfill          $\LipFKa^\emptyset$ is a \,\erm{BGN}\,--\,class on $\,
\ConFKa$ over $\ssp\tfbbR\sp\,$.      \hfill\null          \end{proposition}

\begin{proof} We first note that \math{\tfbbR\in\ConFKa}, and that by
\cite[Corollary 4.1.7, p.\ 85]{FK} every \math{\tilde f\in\LipFKa^\emptyset}
is continuous, trivially having open domain. Hence, for $\ssp\LipFKa^\emptyset$
to be a productive class on \math{\ConFKa} over \math{\tfbbR} in the sense of
\cite[Definitions 4, p.\ 7]{SeBGN}\ssp, for arbitrarily given \math{E\ssp,F\in
\ConFKa} there must be some \math{G\in\ConFKa} with \œ$\ssp
\sigrd G=$ \œ$(\sp\sigrd E\ssp)\vstimes(\sp\sigrd F\ssp)\ssp$ and \math{
(\spp G\sp,E\ssp,\roman{pr}\ar 1\ssp|\,\vecs G\ssp)\ssp,
(\spp G\sp,F\sp ,\roman{pr}\ar 2\ssp|\,\vecs G\ssp)\in\LipFKa^\emptyset}
and such that \linebreak \œ$
(\spp H\sp,E\ssp,f\ssp)\ssp,(\spp H\sp,F\sp,\sp g\ssp)\in\LipFKa^\emptyset\sp
\imply\ssp(\spp H\sp,G\sp,[\,\sp f\sp,\spp g\rbrakf)\in\LipFKa^\emptyset\ssp$
  for all \œ$\ssp f\sp,g\ssp,H\sp$. In view of Corollary \ref{[f,g] is Lip^0}
    above, we may take $\ssp G=E\sacap F\sp$.

In order to establish (1)\œ$\ssp,\ldots\,$(6) of
           \sp\cite[Definitions 4\ssp, p.\ 7]{SeBGN}\ssp, below referred to \linebreak
by \œ$(1)\subtext{BGN}\ssp,\ldots\,(6)\subtext{BGN}\,$, we note the following.
We get (1)$\subtext{BGN}$ from \cite[Proposition 4.3.2, p.\ 99]{FK}\ssp, and
(2)$\subtext{BGN}$ and (3)$\subtext{BGN}$ follow directly from
(\ref{defi LipFKd^k}) and (\ref{defi LipFKa^k}) of
  Constructions \ref{defi Lip ...} above. For (4)$\subtext{BGN}$ with \math{
\Iota = \seq{\,\sp t^{\sp\mminus 1}\sn:t\in\bbR\setminus\snn\{0\}\,}\sp}, note
that we \biggerlineskip2 have \math{|\,\Iota\fvalue s-\Iota\fvalue t\ssp\,|\le
 2\sp\,t^{\sp\mminus 2\,}|\,s-t\,| } when \math{ s\ssp,t\in\bbR } with \math{
|\,s-t\sp\,|<\frac12\,|\ssp t\ssp|\not=0\sp}. Using this, an elementary appeal
to Constructions \ref{defi Lip ...} with details left to the reader gives \œ$
(\sp\tfbbR\sp\,,\sn\tfbbR\sp\,,\sp\Iota\ssp)\in\LipFKa^\emptyset \, $. For
(5)$\subtext{BGN}$\ssp, letting \math{f\sp,\sp g \in
\LipFKa^\emptyset\!\image\snn\{\ssp(\sp\tfbbR\sp\,,F\ssp)\ssp\} } hold with \linebreak
\œ$0\in\dom\sn f=\dom g\ssp$ and \math{ f\sp\fvalue t=g\fvalue t } for \math{
t\not=0\sp}, we should have \math{f=g\sp}. In order to \linebreak get this
indirectly, supposing that \math{f\not=g\sp}, we have \math{f\sp\fvalue 0\not=
g\fvalue 0\sp}, and since \linebreak $\taurd\snn(\sp\deltaAV\KN1\fvalue F\ssp)\ssp$
is point separating, there is \math{ \ell \in
 \taurd\snn(\sp\deltaAV\KN1\fvalue F\ssp) } such that with \œ$\ssp \varphi =
\ell\sp\circ\snn f$ and \math{ \vartheta=\ell\sp\circ\spp g } we have \math{
\varphi\fvalue 0\not=\vartheta\sp\fvalue 0\sp}. Since now \œ$\ssp \varphi\ssp,
\vartheta\in\LLip_{}^{}\ssn$, there is \math{\delta\in\rbb R^+ } with \œ$
\varphi\fvalue s\not=\vartheta\sp\fvalue s \ssp$ for all \math{s} with \math{
\minus\delta<s<\delta }. However, for these \math{ s\not=0 } we also have $
\varphi\fvalue s=\ell\sp\circ\snn f\sp\fvalue s=\ell\sp\circ\spp g\fvalue s =
 \vartheta\sp\fvalue s\ssp$, \,a contradiction.

To get (6)$\subtext{BGN}\ssp$, we have to establish \math{ \tilde a \ssp ,
\widetilde m\in\LipFKa^\emptyset} for an arbitrarily fixed \œ$\ssp E \in$ $
\Con_{FK}a\ssp$ and \math{ \tilde a = (\spp E\sacap E\ssp,\spp E\ssp,
\ssigrd E\ssp)} and \math{\widetilde m=(\spp G\sp,\spp E\ssp,m\sp)} where \math{
G=\tfbbR\sacap E} and \œ$m=\tsigrd E\,$. We explicitly show that \math{
\widetilde m\in\LipFKa^\emptyset\sp}, leaving the similar proof of \œ$\tilde a
\in\LipFKa^\emptyset\ssp$ as an exercise to the reader. Indeed, given any \math{
\Gamma\in\LLip_{}^\emptyset(\sp\deltaAV\KN1\fvalue G\ssp)\sp}, we must get \math{
m\circ\Gamma\in\LLip_{}^\emptyset(\sp\deltaAV\KN1\fvalue E\ssp)\sp}, which in
turn follows if for arbitrarily fixed \œ$\,\ell\in$ $
\taurd(\sp\deltaAV\KN1\fvalue E\ssp)\ssp$ we show (m) that \œ$\,
\ell\sp\circ\spp m\circ\Gamma \in \LLip_{}^{}\ssn$. Putting \math{ \gamma =
\roman{pr}\ar 1\snn\circ\spp\Gamma} and \œ$\, c =
\roman{pr}\ar 2\circ\spp\Gamma\sp$, by Lemma \ref{Lip^0 in EF} above we have \œ$\,
\gamma \in \LLip_{}^\emptyset(\sp\deltaAV\KN1\fvalue\sp\tfbbR\ssp) = \LLip_{}^{}\ssn
$ and \math{ c \in \LLip_{}^\emptyset(\sp\deltaAV\KN1\fvalue E\ssp)\sp}, and
hence $\,\ell\sp\circ\sp c\in\LLip_{}^{}\ssn$. Now, for all \math{t} we have \vskip.3mm $\nKP{2.4}
   \ell\circ m\circ\Gamma\sp\fvalue t
 = \ell\circ m\circ[\,\sp\gamma\ssp,\sp c\rbrakf\!\fvalue t
 = \ell\ssp\fvalue(\sp m\fvalue(\sp\gamma\fvalue t\ssp,\sp c\fvalue t\sp)))$ \vskip.3mm $\nKP{20}
 = \gamma\fvalue t\cdot(\ssp\ell\ssp\fvalue(\sp c\fvalue t\sp))
 = \gamma\fvalue t\cdot(\ssp\ell\sp\circ\sp c\sp\fvalue\sp t\sp)
 = \gamma\cdot(\ssp\ell\sp\circ\sp c\,)\fvalue\spp t\ssp$, \nKP{4.6} whence \vskip.3mm

\noin $\ell\circ m\circ\Gamma=\gamma\cdot(\ssp\ell\sp\circ\sp c\,)\,$. Using $\,
\{\,u\cdotn v:u\ssp,v\in\LLip_{}^{}\sn\}\inc\LLip_{}^{}\ssn$, we get (m)
  above.                                                         \end{proof}

\begin{definitions}\label{defi var, etc.}

For all classes $\sp\tilde f\sp$ we let \vskip.5mm

$\abDeltaFK\tilde f=\LipFKa^\emptyset\vbDelta_\stfbbR\tilde f$ \hfill
     (\sp see \cite[Definitions 7, p.\ 8]{SeBGN}\ssp)\,, \KP{1.3} \vskip.5mm

$\avarFK\tilde f\KP1 = \bigcap\ssp\{\ssp(E\sacap E\ssp,\sp F\sp,\sp g\ssp):
\exi{f\sp,\sp U}\,\tilde f =
(E\ssp,F\sp,f\ssp)\in\ConFKa\ssn^{\times 2.}\sn\times\Univ$ and \vskip.2mm $\mhyppy{22}
f\in(\sp\vecs F\ssp)\,^U$ and $\ssp U\in\taurd E$ and $\ssp g =
\{\ssp(\sp x\ssp,u\ssp,y\sp):$ \vskip.2mm $\mhyppy{25}
\all{\eps\in\rbb R^+,\ssp\ell\in\taurd\sp(\sp\deltaAV\KN1\fvalue F\ssp)}\,
\exi{\delta\in\rbb R^+}\,\all{t\in\bbR}\,$ \par $\null\hfill
|\ssp t\ssp|<\delta\imply
|\,\sp\ell\ssp\fvalue(\sp f\sp\fvalue(\sp x+t\,u\sp)\svs E-
f\sp\fvalue x-t\,y\sp)\svs F\ssp|\le|\ssp t\ssp|\,\eps\,\}\sp\}\,$, \KP{1.3} \vskip.5mm

$\abThetaFK\tilde f=\bigcap\ssp\{\,\tilde g:\exi{E\ssp,F\sp,f\sp,\sp g}\,
\all h\,
\tilde f=(E\ssp,F\sp,f\ssp)\in\ConFKa\ssn^{\times 2.}\sn\times\Univ\ssp$ and \vskip.2mm

$\mhyppy{11}
f\in F^{\sp/\sp E}\sp$ and $\ssp\tilde g =
(E\sacap E\sp\sqcap(\sp\tfbbR\sp\sqcap\tfbbR\ssp)\ssp,\sp F\sp,\sp g\ssp) \in
\LipFKa^\emptyset\sp$ and $\ssp[\ h=$ \vskip.2mm

$\mhyppy{4}
\{\ssp(\sp x\ssp,u\,;s\ssp,t\,;\sp y\ssp):
f\sp\fvalue(\sp x+s\,u\sp)\svs E=(\sp
f\sp\fvalue(\sp x+t\,u\sp)\svs E + (\sp s-t\sp)\,y\sp)\svs F\not=\Univ\sp\,\}$ \vskip.2mm

$\null\hfill
\imply\ssp g\inc h\ssp$ and $\ssp\dom h\inc\dom g\ ]\sp\,\}\,$. \KP2 \vskip.5mm

\noin We say that \math{\tilde f} is {\it directionally
\erm{\,FK\ssp}{\rm a\,}--\,differentiable\ssp} if and only if we have\vskip.5mm\centerline{$
\dom\taurd\tilde f\times\vecs\ssigrd\tilde f\inc\dom\taurd\ssn\avarFK\tilde f
 \not=\Univ\,$.}                                           \end{definitions}

Note above that for example \œ$(\sp x+t\,u\sp)\svs E =
(\sp\ssigrd E\ssp\fvalue(\sp x\ssp,\tsigrd E\ssp\fvalue(\sp t\ssp,u\sp)))$
which would usually be denoted by the ambiguous \q{\œ$\spp(\sp x+t\,u\sp)\sp
$}\sp, assuming it implicitly understood that the linear structure of the
space $E$ is involved in the notation.

In view of \cite[Proposition 4.3.12, p.\ 104]{FK} for \math{\tilde f =
(\spp E\ssp,\spp F\sp,f\ssp)} and \œ$\,U=\dom\sn f\sp$, under the condition
that \math{\avarFK\tilde f\in\LipFKa^\emptyset} holds and that \math{ S \inc
\taurd\sp(\sp\deltaAV\KN1\fvalue F\ssp)} is point separating, we have that \math{
\tilde f} is directionally \erm{\,FK\ssp}{\rm a\,}--\,differentiable if and
only if \œ$f:$ $\deltaAV\KN1\fvalue E\iinc U\to\deltaAV\KN1\fvalue F\ssp$ is $\sp
S\,$--\,differentiable in the sense \cite[Definition 4.3.9, p.\ 103]{FK}\ssp.

Anyway \math{S\sp}--\,differentiability follows from \math{\tilde f} being
directionally \erm{\,FK\ssp}{\rm a\,}--\,differenti- able by \math{
S\inc\taurd\sp(\sp\deltaAV\KN1\fvalue F\ssp)\sp}. Then \math{
\taurd\ssn\avarFK\tilde f} is the function denoted by \sp\q{$\sp\roman df\ssp$}
in \cite[Definition 4.3.9, p.\ 103]{FK}\ssp. For the beginning of the proof,
note the implications \vskip.5mm \centerline{$
\dom\taurd\ssn\avarFK\tilde f \not=\Univ\impss22\taurd\ssn\avarFK\tilde f
\not=\Univ\impss22\avarFK\tilde f \not=\Univ\impss22\exi{E\ssp,F\sp,f\sp,U}$} \vskip.5mm $\null\hfill
\tilde f=(E\ssp,F\sp,f\ssp)\in\ConFKa\ssn^{\times 2.}\sn\times\Univ$ and $
       f\in(\sp\vecs F\ssp)\,^U$ and $\ssp U\in\taurd E\,$. \KP{9.7} \vskip.5mm

\noin In \cite[p.\ 103]{FK} the map \math{
   (\sp\deltaAV\KN1\fvalue(E\sacap E\ssp)\ssp,\deltaAV\KN1\fvalue F\sp,
     \roman df\ssp)} is called the \q{differential} but following
\cite[p.\ 206]{AS} we prefer to call \math{\avarFK\tilde f} the {\it
variation\ssp}, with the notation \œ$ \roman F\,x =
 \{\ssp(\sp\seq{\ssp u\ssp}\ssp,v\sp) : (\sp x\ssp,u\ssp,v\sp) \in
  \taurd\ssn\avarFK\tilde f\sp\,\}\ssp$ reserving the term differential to
referring to the function \math{ g = \seq{\,\ssp\roman F\,x : x \in \domm
  \taurd\ssn\avarFK\tilde f\ssp\,} } only in the case where \math{\roman F\,x}
is linear $(\sp\sigrd E\ssp)\expnota^1.]_{vs}\to\sigrd F\ssp$ for every \math{
  x\in\dom g\sp}.

\begin{lemma}\label{basic bTheta propery}

For every $\,\tilde f$ and for \œ$\,\tilde g=\abThetaFK\tilde f\sp${\rm, }%
either \œ$\,\tilde g=\Univ$ or there are $\,E\ssp,\spp F\sp,\spp f\sp,\sp g$
such that \œ$\,\tilde f=(\spp E\ssp,\spp F\sp,\spp f\ssp)$ and \œ$\,\tilde g =
(\spp E\sacap E\sp\sqcap(\sp\tfbbR\sp\sqcap\tfbbR\ssp)\ssp,\sp F\sp,\sp g\ssp)$
with \œ$\,\tilde f\sp,\sp\tilde g\in\LipFKa^\emptyset$ and $\, \dom g =
\{\ssp(\sp x\ssp,u\,;s\ssp,t\sp) : (\sp x+s\,u\sp)\svs E\ssp,
      (\sp x+t\,u\sp)\svs E\in\dom\sn f\sp\,\}\,${\rm, }and \vskip.3mm\centerline{$
g\fvalue\sp\smb W = ((\sp s-t\sp)^{\sp\mminus 1\,}(\sp
                     f\sp\fvalue(\sp x+s\,u\sp)\svs E
                   - f\sp\fvalue(\sp x+t\,u\sp)\svs E)\svs F)\svs F $}\vskip.3mm

\noin whenever $\,\smb W = (\sp x\ssp,u\,;s\ssp,t\sp)\in\dom g\,$ with $\,
                  s\not=t\ssp$.                                  \end{lemma}

\begin{proof} Fix $\ssp\tilde f\sp$, and suppose that \math{\tilde g\not=\Univ\sp
}. As \math{\bigcap\,\emptyset=\Univ\sp}, then \math{ \{\,\tilde g\ar 1\sn :
(\ssp\tilde g\ar 1)\subtext C\ssp\}\not=\emptyset\sp}, when we let \math{
(\ssp\tilde g\ar 1)\subtext C} denote the formula obtained from the four\ssp-\sp
line formula ``\,$\exists\,\sp E\ssp,$ $F\sp,\spp f\sp,\sp g\sp\,;\,\aall h\,
\ldots\ $'' occurring in the definition of \q{\ssp$\abThetaFK\tilde f\ssp$} in
\ref{defi var, etc.} above by putting there \math{\Symboo\sp\tilde g\ar 1\snnÏ}
in place of \math{\Symboo\tilde gÏ}. Hence, there are \math{E\ssp,\spp F\sp,\spp
f\sp,\sp g} such that we have \œ$\ssp\tilde f=$ \œ$
(\spp E\ssp,\spp F\sp,\spp f\ssp)\in\ConFKa\ssn^{\times 2.}\sn\times\Univ\ssp$
and \math{f\in F^{\sp/\sp E} } and \œ$\ssp\tilde g\ar 1 =
(\spp E\sacap E\sp\sqcap(\sp\tfbbR\sp\sqcap\tfbbR\ssp)\ssp,\sp F\sp,\sp g\ssp)
\in$ $\LipFKa^\emptyset\ssp$, and such that letting \math{h} be the set of all
\math{(\sp x\ssp,u\,;s\ssp,t\,;\sp y\ssp)} with the property that \math{
f\sp\fvalue(\sp x+s\,u\sp)\svs E = (\sp
f\sp\fvalue(\sp x+t\,u\sp)\svs E + (\sp s-t\sp)\,y\sp)\svs F\not=\Univ\sp},
then \math{g\inc h} and \œ$\dom h\inc\dom g\,$. Noting that \math{\tilde f}
uniquely determines $\ssp E\ssp,\spp F\sp,\spp f\sp$, if we prove that the
preceding conditions also uniquely determine \math{g\sp}, we get \math{
\tilde g=\tilde g\ar 1}, and it remains to establish the formulas for \math{
\dom g} and $\ssp g\fvalue\sp\smb W\sp$, and that also \math{
  \tilde f\in\LipFKa^\emptyset\sp}.

Since \math{f\sp\fvalue(\sp x+s\,u\sp)\svs E = (\sp
            f\sp\fvalue(\sp x+t\,u\sp)\svs E + (\sp s-t\sp)\,y\sp)\svs F \not=
\Univ} is equivalent to having \math{ s\ssp,t\in\bbR } and \math{ x\ssp,u \in
\vecs E } and \math{ y\in\vecs F } such that \math{ x+s\,u\ssp,\sp x+t\,u \in
\dom\sn f } \linebreak and  \ also $ \hfill
  f\sp\fvalue(\sp x+s\,u\sp) - f\sp\fvalue(\sp x+t\,u\sp) = (\sp s-t\sp)\,y     \hfill$
hold, \ if \ we \ put 

\noin \œ$ O = \{\ssp(\sp x\ssp,u\ssp,t\sp) : f\sp\fvalue(\sp x+t\,u\sp)\svs E
\not = \Univ\sp\,\} \ssp$ and \math{ Z = \{\ssp(\sp x\ssp,u\,;\sp t\ssp,t\sp)
: (\sp x\ssp,u\ssp,t\sp)\in O\,\} \sp}, and let \math{ h\ar 1} be the function
defined by \math{ \smb W \mapsto (\sp s-t\sp)^{\sp\mminus 1\,}(\ssp
                      f\sp\fvalue(\sp x+s\,u\sp) - f\sp\fvalue(\sp x+t\,u\sp)) }
on the set \math{W} of all \math{ \smb W=(\sp x\ssp,u\,;s\ssp,t\sp) } with \math{
s\not=t } and \math{ \{\ssp(\sp x\ssp,u\sp)\ssp\}\timesn\{\ssp s\ssp,t\ssp\}
\inc O }, we have $ h = \{\ssp(\sp x\ssp,u\,;\sp t\ssp,t\sp) :
(\sp x\ssp,u\ssp,t\sp)\in O\,\}\times\vecs F\cupss22 h\ar 1\ssp$ with \math{
  \dom h=W\cupss11 Z }.

From the preceding observations we already get the formulas for \math{\dom g}
and $g\fvalue\sp\smb W\sp$, provided that \math{g} is known to be uniquely
determined. To prove this in- \linebreak directly, supposing that there is
another \math{g\ar 1} with the properties of \math{g\sp}, there \linebreak is
\math{ \smb W = (\sp x\ssp,u\,;\sp t\ssp,t\sp) \in Z } with \œ$\ssp
g\fvalue\sp\smb W \not = g\ar 1\KN1\fvalue\sp\smb W\sp$. Taking \math{ G =
E\sacap E\sp\sqcap(\sp\tfbbR\sqcap\tfbbR\ssp) } and \linebreak \œ$ \gamma =
\seq{\,(\sp x\ssp,u\,;\sp t+s\ssp,t\sp):s\in\ssbb03 R\,} \,$, and considering
\math{ c = g\circ\gamma } and \math{ c\ar 1=g\ar 1\snn\circ\gamma }, since we
have \math{ \gamma \in \LLip_{}^\emptyset G }, and since \math{ c\fvalue s =
c\ar 1\KN1\fvalue s } for \math{ s\not=0 \sp}, by (5)$\subtext{BGN}$ we get \math{
c = c\ar 1 }, and hence in particular $\, g\fvalue\sp\smb W =
g\circ\gamma\fvalue 0 = c\fvalue 0 = c\ar 1\KN1\fvalue 0 =
g\ar 1\sn\circ\sp\gamma\spp\fvalue\spp 0 = g\ar 1\KN1\fvalue\sp\smb W\sp$. \vskip.25mm

Finally, to get \math{\tilde f\in\LipFKa^\emptyset\sp}, if \math{ f=\emptyset }
holds, the assertion is trivial directly by definition. Otherwise, we
arbitrarily fix any \œ$\ssp(\sp x\ar 0\ssp,\sp y\ar 0)\in f\sp$, and note that
for all $\ssp x$ we then have \œ$\, f\sp\fvalue x = (\ssp y\ar 0 +
   g\fvalue(\sp x\ssp,(\sp x-x\ar 0)\svs E\,;\sp 0\,,\minus 1\sp))\svs F\,$.
Further, putting \œ$\ssp G = $ \linebreak \œ$
E\sacap E\sacap(\sp\tfbbR\sacap\tfbbR\ssp)\ssp$ and \math{ \gamma =
\seq{\,(\sp x\ssp,x-x\ar 0\,;\sp 0\,,\minus 1\sp):x\in\vecs E\ssp\,}\sp}, and
using \œ$\ssp G=$ \linebreak \œ$
E\sacap E\sp\sqcap(\sp\tfbbR\sp\sqcap\tfbbR\ssp)\ssp$ and (2)$\subtext{BGN}$
and (3)$\subtext{BGN}$ and (6)$\subtext{BGN}\,$, and either
\cite[Proposition 6\ssp(b)\ssp, pp.\ 7\,--\,8]{SeBGN} or
Corollary \ref{[f,g] is Lip^0} above, we successively get \vskip.5mm

$\KP9 (\spp E\ssp,\spp E\sacap E\ssp,\seq{\,
(\sp x\ssp,\minus x\ar 0):x\in\vecs E\ssp\,}\ssp)\in\LipFKa^\emptyset\KP2$, \vskip.3mm

$\KP9(\spp E\ssp,\spp E\ssp,\seq{\,
              x-x\ar 0\sn:x\in\vecs E\ssp\,}\ssp)\in\LipFKa^\emptyset\KP2$, \vskip.3mm

$\KP9(\spp E\ssp,\spp E\sacap E\ssp,\seq{\,
     (\sp x\ssp,x-x\ar 0):x\in\vecs E\ssp\,}\ssp)\in\LipFKa^\emptyset\KP2$, \vskip.3mm

$\KP9(\spp E\ssp,\spp G\sp,\sp\gamma\ssp)\in\LipFKa^\emptyset\KP2$, \ 
$(\spp E\ssp,\spp F\sp,\sp g\spp\circ\sp\gamma\ssp)\in\LipFKa^\emptyset\KP2$, \vskip.3mm

$\KP9(\spp E\ssp,\spp F\sacap F\sp,
[\,\sp\vecs E\times\{\ssp y\ar 0\}\ssp,\ssp g\spp\circ\sp\gamma\rbrakf)
\in\LipFKa^\emptyset\KP2$, \ $\tilde f\in\LipFKa^\emptyset\,$.   \end{proof}

In fact, by \cite[Corollary 4.5.6, p.\ 137]{FK} in
Lemma \ref{basic bTheta propery} for all $\tilde f$ we even have either \œ$\sp
\abThetaFK\tilde f=\Univ$ or \œ$\sp\tilde f\in\LipFKa^{1.}\,$. In the case
where \œ$\sp\tilde f=(E\ssp,F\sp,f\ssp)$ with \œ$\sp\abThetaFK\tilde f \not =
\Univ\,$, in \cite[p.\ 105\,ff.]{FK} the function $\ssp
\taurd\ssn\abThetaFK\tilde f\sp$ is denoted by \sp\q{$\bar\vartheta f\ssp$}.

Without formulating it as an explicit lemma, we mention that similarly as in
the preceding proof, see also \cite[Proposition 9, p.\ 8]{SeBGN}\ssp, one
  deduces that \vskip.5mm

{\it for all $\,\tilde f$ and for \œ$\,\tilde g = \abDeltaFK\tilde f\sp${\rm, }%
either \œ$\,\tilde g=\Univ$ holds{\sp\rm, }or there are $\,E\ssp,\spp F\sp,\spp
f\sp,\sp g$ such that \œ$\,\tilde f=(\spp E\ssp,\spp F\sp,\spp f\ssp)$ and \œ$\,
\tilde g = (\spp E\sacap E\sp\sqcap\tfbbR\sp\,,\sp F\sp,\sp g\ssp)$ with \œ$\,
\tilde f\sp,\sp\tilde g\in\LipFKa^\emptyset$ and \œ$\, \dom g = $ $
\{\ssp(\sp x\ssp,u\ssp,t\sp):(\sp x+t\,u\sp)\svs E\in\dom\sn f\sp\,\} \,
${\rm, }\,and \vskip.3mm\centerline{$
 g\fvalue\smb Z = (\ssp t^{\sp\mminus 1\,}(\sp
         f\sp\fvalue(\sp x+t\,u\sp)\svs E - f\sp\fvalue x\sp)\svs F)\svs F $ } \vskip.3mm \noin
whenever $\,\smb Z = (\sp x\ssp,u\ssp,t\sp)\in\dom g\,$ with $\,t\not=0\,$. } \vskip.5mm

In the case \math{\tilde g\not=\Univ \sp}, by direct appeals to
\cite[Definitions 7, p.\ 8]{SeBGN} we even see that \math{ \tilde f \in
\CalDBGN^{1.}(\sp\LipFKa^\emptyset\ssp,\sn\tfbbR\ssp)\sp}. Using this, in view
of Proposition \ref{Lip0 is BGN} above, from
\cite[Proposition 10\ssp, p.\ 9]{SeBGN} we further get the following

\begin{corollary}\label{recu for BGN Lip0}

$\null\hfill\all{\tilde f\ssp,\sp k}\,\tilde f \in
\CalDBGN^{k\ssp+\sp 1.}(\sp\LipFKa^\emptyset\ssp,\sn\tfbbR\ssp)\sp\equivv\sp
\abDeltaFK\tilde f\in\CalDBGN^k(\sp\LipFKa^\emptyset\ssp,\sn\tfbbR\ssp)\,$.
 \hfill\null                                                 \end{corollary}

\begin{proposition}\label{part Lip k recu}

For all $\ssp\tilde f\ssp,\sp k\sp$ the implications \vskip.3mm\centerline{$
\tilde f$ is directionally \erm{\,FK\ssp}{\rm a\,}--\,differentiable and $\,
\avarFK\tilde f\in\LipFKa^k$} \vskip.3mm $\mhyppy{13.2}
\imply\,\tilde f\in\LipFKa^{k\ssp+\sp 1.}\,
\imply\,\abThetaFK\tilde f\in\LipFKa^k \KP{24.6}$ hold.    \end{proposition}

\begin{proof} In view of the lines 4\,--\,7 after
Definitions \ref{defi var, etc.} above, and also noting that \œ$\sp\LipFKa^k=
\emptyset$ if \œ$k\not\in\infty\ssp\yplus\sp$, and that \œ$\sp
\LipFKa^\inftyy=\{\,\tilde f:\all{k\in\bbNo}\,\tilde f\in\LipFKa^k\ssp\}\,$,
the as- serted implications follow from \cite[Definition 4.3.13,
Theorem 4.3.24, Corollary 4.5.6, pp.\ 104\,--\,106, 110\,--\,111,
                                      137\,--\,138]{FK}\ssp.     \end{proof}

Generally defining \vskip.3mm \centerline{$
\bmii8G\sp\Circ\spp\bmii8F\, = \bigcap\ssp\{\ssp(X\sp,Z\sp,g\sp\circ\snn f\sp)
 : \exi{\sp Y\sn\in\Univ\sp}\,\bmii8F\, =
   (X\sp,Y\spp,f\sp)\text{ and }\bmii8G\,=(\ssp Y\spp,Z\sp,g\sp)\ssp\}\,$,}\vskip.3mm

\noin from \cite[Summary 2.4.4\ssp(iii)\ssp, (6\œ$)\imply($3)\ssp, p.\ 51]{FK}
we get (1)\ssp, and directly from (\ref{defi LipFKa^k}) and
(\ref{defi LipFKd^k}) of Constructions \ref{defi Lip ...}\ssp, noting also
Lemma \ref{FKd eqv FKa} above, we get (2) in the next

\begin{proposition}\label{basic Lip^k properties}

For all $\,E\ssp,F\sp,\tilde f\ssp,\ssp\tilde g\ssp,\sp k\ssp,\ssp\ell\,$ it
holds that \vskip.5mm

{\rm(1)} \ $E\ssp,F\in\ConFKa$ and $\,\ell\in\Cal L\,(E\ssp,F\ssp)\impss22
           (E\ssp,F\sp,\sp\ell\,)\in\LipFK_a^\inftyy\,,$\vskip.3mm

{\rm(2)} \ $\tilde f\ssp,\ssp\tilde g\in\LipFK_a^k \impss22
          \tilde g\spp\Circ\snn\tilde f\in\LipFK_a^k$ or $\,
           \tilde g\spp\Circ\snn\tilde f=\Univ\,$.         \end{proposition}

\begin{theorem}\label{main Thm}

The equality $\,\LipFKa^k=\CalDBGN^k(\sp\LipFKa^\emptyset\ssp,\sn\tfbbR\ssp)\,
$ holds for all $\ssp k\ssp$.                                  \end{theorem}

\begin{proof} Let \œ$\ssp\roman C^{\,k} = \CalDBGN^k(\sp\LipFKa^\emptyset\ssp,\sn
\tfbbR\ssp)\,$, to simplify the notations a bit. Since we \linebreak have \œ$\,
\roman D\,^k=\emptyset$ if \œ$k\not\in\infty\ssp\yplus\sp$, and since \œ$\sp
\roman D\,^\infty=\{\,\tilde f:\all{k\in\bbNo}\,\tilde f\in\roman D\,^k\ssp\}$
when $\sp\roman D\,^k$ denotes either of $\LipFKa^k$ and $\sp\roman C^{\,k}\sp
$, it suffices to prove \œ$\ssp\all{k\in\bbNo}\,(k)\subtext A$ by induction,
letting $(k)\subtext A$ mean that $\LipFKa^k=\roman C^{\,k}$ holds. We have $
(\emptyset)\subtext A$ trivially.

Before considering the inductive step, we note the auxiliary result ($*$) that
$\sp\tilde f$ is directionally \erm{\,FK\ssp}{\rm a\,}--\,differentiable
whenever \œ$\abDeltaFK\tilde f\not=\Univ$ holds. Indeed, then \œ$
\abDeltaFK\tilde f\in\LipFKa^\emptyset$ and there are $E\ssp,F\sp,f$ with \œ$
\tilde f=(E\ssp,F\sp,f\ssp)\,$. For arbitrarily fixed \œ$x\in\dom\sn f$ and \œ$
u\in\vecs E\,$, putting \œ$\,c = \taurd\ssn\abDeltaFK\tilde f\circ\seq{\,
(\sp x\ssp,u\ssp,t\sp):t\in\ssbb04 R\,}\,$, we have\linebreak \œ$
(\sp\tfbbR\sp\,,\spp F\sp,\sp c\,)\in\LipFKa^\emptyset$ with \œ$0\in\dom c$
and \œ$t^{\sp\mminus 1}\sp(\sp f\sp\fvalue(\sp x+t\,u\sp) - f\sp\fvalue x\sp)=
c\fvalue t$ when \œ$0\not=$ \œ$t\in\dom c\,$. Then having \œ$
(\sp\tfbbR\sp\,,\sn\tfbbR\sp\,,\ssp\ell\sp\circ\spp c\,)\in\LipFKa^\emptyset$
for all \œ$\ssp\ell\in\taurd(\sp\deltaAV\KN1\fvalue F\ssp)\,$, a glance at
Definitions \ref{defi var, etc.} shows that we have \œ$
(\sp x\ssp,u\ssp,c\fvalue 0\sp)\in\taurd\ssn\avarFK\tilde f\sp$. From this we
see that \œ$\dom\taurd\ssn\avarFK\tilde f = \dom\sn f\times\vecs E\,$, hence
that \œ$\dom\taurd\tilde f\times\vecs\ssigrd\tilde f \inc
\dom\taurd\ssn\avarFK\tilde f$ \linebreak \œ$\not=\Univ\sp$ holds, and
consequently that $\tilde f$ is directionally
\erm{\,FK\ssp}{\rm a\,}--\,differentiable. Note \linebreak that $\sp
\taurd\ssn\avarFK\tilde f\sp$ is a function since $\sp
\taurd(\sp\deltaAV\KN1\fvalue F\ssp)$ is point separating.

Now assuming $(k)\subtext A$ with \œ$k\in\bbNo\ssp$, to get \œ$
(\sp k+1\adot)\subtext A\ssp$, first let \œ$\sp
\tilde f\in\LipFKa^{k\ssp+\sp 1.}\sp$. By Proposition \ref{part Lip k recu}
then \œ$\abThetaFK\tilde f\in\LipFKa^k\,$. Since for \œ$E=\ssigrd\tilde f$ and
\œ$H=E\sacap E\sp\sqcap\tfbbR$ and \œ$ \tilde\Iota\ar 3 =
(\sp H\sp,E\sacap E\sp\sqcap(\sp\tfbbR\sp\sqcap\tfbbR\ssp)\ssp,
\seq{\,(\sp x\ssp,u\,;\sp s\ssp,0\sp):\smb X=(\sp x\ssp,u\ssp,s\sp) \in
\vecs H\sp\,}\sp)$ we have $\tilde\Iota\ar 3$ a continuous linear map with \œ$
\abDeltaFK\tilde f=\abThetaFK\tilde f\Circ\tilde\Iota\ar 3\not=\Univ\,$, by
$(k)\subtext A$ and Proposition \ref{basic Lip^k properties} we obtain \œ$
\abDeltaFK\tilde f\in\LipFKa^k\inc\roman C^{\,k}\sp$, whence \œ$\sp\tilde f\in
\roman C\,^{k\ssp+\sp 1.}$ follows by
Corollary \Biggerlineskip1 \ref{recu for BGN Lip0} above. Conversely, letting
\œ$\sp\tilde f\in\roman C\,^{k\ssp+\sp 1.}\sp$, by
Corollary \ref{recu for BGN Lip0} again we have \œ$\abDeltaFK\tilde f\in$ \œ$
\roman C^{\,k}\inc\LipFKa^k\,$, in view of $(k)\subtext A\,$. For \œ$
\tilde\Iota\ar 2=(\sp G\sp,\spp H\sp,\seq{\,(\sp x\ssp,u\ssp,0\sp):\smb X =
(\sp x\ssp,u\sp)\in\vecs G\sp\,}\sp)$ \Biggerlineskip1 with \œ$G=E\sacap E$
having $\tilde\Iota\ar 2$ a continuous linear map with \œ$\avarFK\tilde f =
\abDeltaFK\tilde f\Circ\tilde\Iota\ar 2$ \Biggerlineskip1 \œ$\not=\Univ\,$,
again by Proposition \ref{basic Lip^k properties} we obtain \œ$\avarFK\tilde f
\in\LipFKa^k\,$. Now \œ$\sp\tilde f\in\LipFKa^{k\ssp+\sp 1.}$ fol- lows by
($*$) and Proposition \ref{part Lip k recu} above.               \end{proof}

\begin{remark}

In \cite{SeBGN}\ssp, we used
\cite[Lemma 50, Corollary 51, pp.\ 26\,--\,27]{SeBGN} as an auxil- iary tool
when establishing \cite[Theorem 52]{SeBGN} corresponding to
Theorem \ref{main Thm} above. Here we did not need to do this kind of work
since the corresponding one is already done in sufficient detail in
\cite[Lemma 4.1.5, Proposition 4.3.11, Theorem 4.5.4, pp.\ \nolinebreak 84,
104, 136\,--\,137]{FK}\ssp.                                     \end{remark}

% ----------------------------------------------------------------------------

\end{document}